\documentclass[12pt]{amsart}
\usepackage[english]{babel}
\usepackage[normalem]{ulem}
\usepackage{amsthm,amsmath,amsfonts,amssymb}
\usepackage{color}
\usepackage{nicefrac}
\usepackage{hyperref}
\usepackage{url}
\usepackage{mathtools}
\mathtoolsset{showonlyrefs}
\usepackage{anysize}
\usepackage{geometry}
\marginsize{2.5cm}{2.5cm}{2.5cm}{2.5cm}
\numberwithin{equation}{section}
\theoremstyle{plain}
\newtheorem{theorem}{Theorem}[section]
\newtheorem{lemma}[theorem]{Lemma}
\newtheorem{proposition}[theorem]{Proposition}
\newtheorem{remark}[theorem]{Remark}
\newtheorem{corollary}[theorem]{Corollary}
\newtheorem{definition}[theorem]{Definition}
\newcommand{\ud}{\mathrm{d}}
\newcommand{\tv}{\mathrm{d}_{\mathrm{TV}}}
\newcommand{\fa}{\mathfrak{a}}
\newcommand{\fb}{\mathfrak{b}}
\newcommand{\ii}{\mathsf{i}}
\newcommand{\e}{\varepsilon}
\newcommand{\lqq}{\leqslant}
\newcommand{\gqq}{\geqslant}

\begin{document}
\title[Cut-off phenomenon in the total variation and Wasserstein distances]{Profile cut-off phenomenon for the ergodic Feller root process}
\author{Gerardo Barrera}
\address{Center for Mathematical Analysis, Geometry and Dynamical Systems, Mathematics Department, Instituto Superior T\'ecnico, Universidade de Lisboa, 1049-001, Av. Rovisco Pais 1, 1049-001 Lisboa, Portugal.}
\email{gerardo.barrera.vargas@tecnico.ulisboa.pt}
\author{Liliana Esquivel}
\address{Department of Mathematical Sciences,Universidad de Puerto Rico, Recinto Mayag\"uez, 00681, Mayag\"uez, Puerto Rico.}
\email{liliana.esquivel@upr.edu}

\begin{abstract}
The present manuscript is devoted to the study of the convergence to equilibrium as the noise intensity $\e>0$ tends to zero for ergodic random systems out of equilibrium driven by multiplicative non-linear noise  of the type
\begin{align*}
\ud X^\e_t(x) = (\fb-\fa X^\e_t(x))\ud t+\e \sqrt{X^\e_t(x)}\ud B_t, \quad 
X^{\e}_0(x) = x, \quad   t\gqq 0,
\end{align*}
where $x\gqq 0$, $\fa>0$ and $\fb>0$ are constants, and $(B_t)_{t \gqq 0}$ is a one dimensional standard Brownian motion. 
More precisely, we show the strongest notion of asymptotic profile cut-off phenomenon in the total variation distance and in the renormalized Wasserstein distance when $\e$ tends to zero with
explicit cut-off time, explicit time window, and explicit profile function.
In addition, asymptotics of the so-called mixing times are given explicitly.
\end{abstract}
\markboth{Abrupt thermalization for the Cox--Ingersoll--Ross (CIR) Model}{Asymptotic profile cut-off phenomenon}

\maketitle


\section{\textbf{Introduction}}\label{sec:introd}
The concept of the asymptotic cut-off phenomenon in the total variation distance was coined by D.~Aldous and P.~Diaconis in the
context of Markov chains models of card-shuffling, Ehrenfests’ urn and random transpositions, see~\cite{Aldous,Diaconis1996}.
Very generally, it refers to the asymptotically abrupt dynamical transition (it is also known as drastic/sharp/step/switch convergence) of the current state out of equilibrium of the Markov chain to its dynamic equilibrium in a cut-off window around the so-called mixing time as a function of the deck size.
In other words, the distance to equilibrium behaves as an approximate step function that remains close to its maximal value (diameter) for a while, and then suddenly drops to zero as the running time parameter reaches a critical threshold, which corresponds to the mixing time.
Since the state space of card-shuffling Markov chains models is finite, the preceding convergence to equilibrium is naturally measured in terms of the total variation distance. 
This dynamical phase transition has been observed in a broad spectrum of random models.
For instance, in the discrete state space the asymptotic cut-off phenomenon has been showed for
shuffling cards Markov dynamics~\cite{Aldous,Diaconis1996}, random walks on the $n$-dimensional hypercube~\cite{PDIbook,Levin2009}, birth and death Markov chains~\cite{BarreraJaviera2,Smith2017},
sparse Markov chains~\cite{Bordenave}, Glauber dynamics~\cite{Ding2009},  
SEP dynamics~\cite{Gantert20231,Lacoin2017}, SSEP dynamics~\cite{Goncalves2023}, TASEP dynamics~\cite{Elboim2022}, random walks in random regular graphs~\cite{Lubetzky2010}, mean-zero field zero-range process~\cite{Merle2019}, averaging processes~\cite{Quattropani2023}, and
sampling chains~\cite{BAmax,Lachaud2005}.
For more general Markov processes taking values in
continuous state-spaces there are relatively few results
showing the asymptotic cut-off phenomenon. They include linear
SDEs driven by 
additive Brownian noise and more general 
additive L\'evy perturbations~\cite{BHChaos,BHPWA,
BarreraLiu2021,BPEJP2020,BarreraJaviera3,Lachaud2005}, non-linear over-damped and under-damped Langevin dynamics driven by 
additive Brownian motion and more general 
additive L\'evy perturbations~\cite{Barrera2018,BarreraEJP2021,BPHWASD,BJJSP,BJAAP,LeeRamilInsuk2023}, linear heat and wave SPDEs driven by 
additive  L\'evy noise~\cite{BPHSPDE},
multivariate geometric Brownian motion~\cite{BPHGBM},
Dyson--Ornstein--Uhlenbeck process~\cite{Boursier}, the biased
adjacent walk on the simplex~\cite{Labbe2022}, Brownian motion on families of compact
Riemannian manifolds~\cite{Meliot}.
Recently, the asymptotic cut-off has been studied in: coagulation-fragmentation systems~\cite{Murray}, machine learning for neural networks~\cite{Avelin}, quantum Markov chains~\cite{Kastoryano12},
open quadratic fermionic systems~\cite{Vernier20},
chemical kinetics~\cite{BOK10}, viscous energy shell model~\cite{BHPPshell},
and random quantum circuits~\cite{Oh}.

In the present manuscript, we are interested in the asymptotic cut-off phenomenon with complexity positive parameter $\e$ in the weak noise limit $\e\to 0^+$ for 
the uniquely ergodic random dynamics $X^{\e,x}:=(X^{\e}_t(x))_{t\gqq 0}$ given by the unique strong solution of the stochastic differential equation (for short SDE) with 
\textit{multiplicative non-linear noise}
\begin{equation}\label{SDE}
\ud X^\e_t(x) = (\fb-\fa X^\e_t(x))\ud t+\e \sqrt{X^\e_t(x)}\ud B_t, \quad 
X^{\e}_0(x) = x, \quad   t\gqq 0,\, x\gqq 0,
\end{equation}
where $B:=(B_t)_{t\gqq 0}$ is a one dimensional standard Brownian motion, and $\fa$ and $\fb$ are positive constants.
The Fokker--Planck equation associated to~\eqref{SDE} is analyzed in 1951 by W.~Feller in~\cite{Feller1951} to study singular diffusion problems. It is then referred as  Feller’s square root process.
The model~\eqref{SDE} is introduced in 1985 by 
J.~Cox, J.~Ingersoll and 
S.~Ross  to describe the evolution of instantaneous interest rate at time $t$, $X^{\e}_t(x)$, with mean reversion
of the interest rate towards the long-term value $\nicefrac{\fb}{\fa}$ with the speed of adjustment $\fa$. 
Nowadays, it is known as a Cox--Ingersoll--Ross (CIR) process, see~\cite{Coxetalt1985}.
In Section~1.5 in~\cite{Boursier}
it is shown that cut-off phenomenon holds true, as the number of sampling increases,  for
a particular CIR model (Bessel process), which can be written as a square of an Ornstein--Uhlenbeck process.
Furthermore, singular SDEs of the type~\eqref{SDE} 
emerge in the model of evolution of chemical reactions and population dynamics as diffusion approximations of res-caled Markov jump processes, see for instance~\cite{Bakeretalt2018,EthierKurtz,Shen} and the references therein.
The stochastic process~\eqref{SDE}  belongs to the class of continuous-state branching processes with immigration (a.k.a. CBI processes).
The so-called fluid limit of~\eqref{SDE} (i.e. the deterministic ordinary differential equation obtained as $\e \to 0^+$) may be interpreted as a mean-field approximation and
the stochastic term $\e \sqrt{X^\e_t(x)}\ud B_t$
may be interpreted as
the demographic or internal noise. 
The study of this case will be relevant in the understanding of the cut-off phenomenon for a wide class of SDEs with non-linear coercive deterministic component and driven by multiplicative noise. Moreover, it also shows the complexity that brings to the problem the multiplicative noise.

By~Theorem~5.1~in~\cite{Fu} or Theorem~3.2~in~\cite{Ikeda} we have that~\eqref{SDE} has a unique strong solution in a filtered complete probability space  $(\Omega,\mathcal{F}, (\mathcal{F}_t)_{t\gqq 0},\mathbb{P})$ satisfying the usual conditions
where the standard Brownian motion $B$ is defined, and denote by $\mathbb{E}$ the expectation with respect to the probability measure $\mathbb{P}$.
We stress that the density of the marginal $X^{\e}_t(x)$  can be
expressed in terms of Bessel functions, 
see~\cite{Coxetalt1985,Feller1951}.
Since we are interested in the limit  $\e\to 0^+$,
without loss of generality we can assume that the so-called ergodic Feller regime holds true, i.e, $\e\in (0,\sqrt{2\fb})$. 
In the aforementioned ergodic Feller regime, the random dynamics given by~\eqref{SDE} is uniquely ergodic and its invariant probability measure $\mu^\e$ possesses a Gamma distribution with parameters
$\nicefrac{2\fb}{\e^2}$ and $\nicefrac{2\fa}{\e^2}$.
Since~\eqref{SDE} belongs to the class of affine diffusions, the marginal $X^{\e}_t(x)$ converges in distribution to $X^{\e}_\infty$ as $t$ tends to infinity. Moreover, the law of $X^{\e}_\infty$ does not depend on the initial datum $x$. The preceding convergence can be improved to be valid in the total variation distance.
For more details, we refer to~\cite{Coxetalt1985,Feller1951,Friesen,Jinetalt2013}.

We stress that in general, convergence in distribution does not imply convergence in total variation distance. For instance, the celebrated
DeMoivre--Laplace Central Limit Theorem does not hold true in the total variation distance, nevertheless, it holds true for other distances such as the 
Wasserstein distance, see Theorem~1 in~\cite{Bobkov 2013}.
Even for absolutely continuous (smooth) sequence of distributions, the convergence in distribution does not imply convergence in total variation distance, see for instance Fig.1 in~\cite{Chae} or Lemma~1.17 in~\cite{BarreraEJP2021}. 
In fact, the convergence in the total variation distance is very strong and hence it cannot be expected
from just the convergence in distribution without additional structure. 

Let $(X_n)_{n\in \mathbb{N}}$ be a sequence of random variables defined in a probability space $(\Omega, \mathcal{F},\mathbb{P})$.
In addition, let $X$ be a random variable defined in $(\Omega,\mathcal{F},\mathbb{P})$.
Assume that 
for each $n\in \mathbb{N}$, the distribution of the random variable $X_n$ is absolutely continuous with respect to the Lebesgue measure on $\mathbb{R}$ and its density $f_n$ is unimodal. 
Moreover, assume that the random variable $X$ has also a unimodal density $f$ with respect to the Lebesgue measure on $\mathbb{R}$. 
Then Theorem~2.3 (Ibragimov) 
in~\cite{DasGupta} or Lemma~3 in~\cite{Ibragimov} yield that the following statements are equivalent:
\begin{itemize}
\item[(i)] The random sequence $(X_n)_{n\in \mathbb{N}}$ converges in distribution to $X$ as $n$ tends to infinity.
\item[(ii)] The random sequence $(X_n)_{n\in \mathbb{N}}$ converges in the total variation distance to $X$ as $n$ tends to infinity.
\end{itemize}
Nevertheless, showing that the distribution of a random variable has a unimodal density is in general a difficult task, see Khinchin's Theorem in Theorem~4.5.1 
of~\cite{Lukacs}.
Without assuming that for each $n\in \mathbb{N}$, $f_n$ and $f$ are unimodal,
in~\cite{Boos,Sweeting} there are sufficient conditions on $(f_n)_{n\in \mathbb{N}}$ so that (i) is equivalent to (ii). However, such conditions requires to verify that the sequence $(f_n)_{n\in \mathbb{N}}$ is equicontinuous, which typically is a challenging task.

In this manuscript, 
we also present a Parseval–Plancherel–Fourier based approach to total variation convergence,
which may be of independent interest,
see Proposition~\ref{lem:fouriertv} and Proposition~\ref{prof:parseval} in Appendix~\ref{apendtv}.
By imposing (a) point-wise convergence of the characteristic functions, (b) integrability of the characteristic functions, and (c) uniform integrability of the characteristic functions, we then have convergence in total variation distance for the corresponding distributions.
We note that (a) is a minimal requirement to study convergence in total variation distance.
To verify~(b) and~(c) we point out that it is not needed explicit formulas for the characteristic functions, only 
(suitable) uppers bounds for the modulus of them are required.
We remark that there are sequences of random variables for which convergence in total variation distance holds true, nevertheless, (b) and (c) are not valid, see Remark~\ref{rem:examplet} in Appendix~\ref{apendtv}.

Bearing all this in mind, we are interested in the quantitative behavior of the  convergence to equilibrium
of the out of equilibrium marginals of the SDE~\eqref{SDE} in the total variation distance and in the Wasserstein distance.
The quantitative analysis of the convergence to equilibrium in a suitable distance or divergence of interest for ergodic Markov processes, in particular SDEs, is typically a difficult task. 
By an analytical approach
it typically requires a refined knowledge of the spectrum (roughly speaking, eigenvalues and their eigenfunctions) for the so-called generator of the SDE, which is an infinitely dimensional operator. In addition, Poincar\'e type inequalities, Log-Sobolev type inequalities, relative entropy inequalities, Lyapunov functions, Kantorovich potentials estimates and/or
Talagrand’s transportation inequalities
are needed to hold true in order to bound the carr\'e du champ  in the Bakry--\'Emery theory, see~\cite{Ledoux}.
We note that for short times, the law of the random dynamics starting from a deterministic initial datum is very far from the invariant probability measure and hence estimates using the aforementioned techniques are far to be optimal.
Alternative approaches are offered by the so-called Stein's method~\cite{Chenbook,Kusuoka,Ley,Nourdin} and probabilistic couplings methods~\cite{Lindvall}. Nevertheless, in general only upper bounds for such convergence can be obtained. In addition, coupling techniques for systems driven by multiplicative noises are more involved comparing with systems driven by additive noises as the SDEs aforementioned in the introduction.

Historically, the exponential quantitative analysis of the convergence  to equilibrium for Markov processes provides upper bounds of the type 
$Ce^{-\delta t}$ for a suitable distance between the law of the dynamics (marginal) at time $t$ and its corresponding invariant probability measure.
The positive \textit{ergodicity constants} (clearly not unique) $C$ and $\delta$ may depend on the complexity parameter (here denoted by $\e$), the initial datum of the dynamics and/or the dimension of the underlying state space.
Therefore, generically, it is a difficult task to  carry out the dependence with respect to the complexity parameter, initial datum and/or dimension, see~\cite{Meyn}. Numerical computations of ergodicity constants for stochastic dynamics have been studied in~\cite{Li2020}.
The constant $-\delta$ is related with the spectral gap of the generator of the Markov process.
As a consequence, the computation and/or estimation of such constants are typically done case by case.
In addition, when the distance can be written in dual-formulation (duality), it is possible to obtain lower bounds by a cunning choice of a distinguish statistic (observable), see for instance the Wilson method for Markov chains given in Proposition~7.7 of~\cite{Levin2009}.
Nevertheless, in general, we stress that lower bounds for the convergence  to equilibrium are typically out of reach.
Recently,
a criterion to obtain lower bounds on the rate of convergence to equilibrium in the so-called
$f$-variation of a continuous-time ergodic Markov process is establishing
in~\cite{Bresar}.

The objective of this manuscript is two-fold.
\begin{enumerate}
\item[(1)] Provide a robust technique to show convergence to equilibrium in the total variation distance and in the Wasserstein distance, see Subsection~\ref{sub:definitions} for the precise definition of the distances.
We remark that the Parseval--Plancherel--Fourier based approach used in this paper to study the asymptotic cut-off phenomenon in total variation distance for~\eqref{SDE} is robust and it 
is not based on explicit formulas for its marginals and/or the aforementioned analytical methods or functional inequalities.
It may be extended to more general  stochastic processes driven by general perturbations under suitable information of the characteristic functions, i.e., point-wise convergence, integrability, and uniform integrability when $\e\to 0^+$  of the spatial Fourier transforms are valid, see Proposition~\ref{lem:fouriertv} and Proposition~\ref{prof:parseval}~in Appendix~\ref{apendtv}.
In particular, after identifying the Central Limit Theorem when $\e\to 0^+$ for the invariant probability measure, that is, after a suitable shift and scaling, we obtain 
the strong notion of asymptotic profile cut-off phenomenon when $\e\to 0^+$ with
explicit cut-off time, explicit time window, and explicit profile function in terms of the classical error function, see Theorem~\ref{mainresult} below.
For the Wasserstein distance, the same holds true whenever uniform integrability when $\e\to 0^+$ of the corresponding $p$-th moments is valid. In addition, the profile function has an explicit exponential shape,
see Theorem~\ref{mainresultW} below.
We point out that
explicit profiles are typically out of reach and that
not all one-dimensional stable stochastic systems exhibit cut-off phenomenon, see 
Theorem~1.5 (Gradual convergence) in~\cite{BCostaJara}.
\item[(2)] Provide explicit asymptotics of the so-called mixing times in the total variation distance and in the Wasserstein distance, see Corollary~\ref{cor:tvmix} and Corollary~\ref{cor:mxw} below.
We emphasize that the asymptotics of the mixing times do not rely on explicit knowledge of optimal ergodicity constants aforementioned, which are generically a difficult task.
\end{enumerate}

The rest of the manuscript is organized as follows.
In Section~\ref{sub:definitions}, we review the necessary background about the Wasserstein distance and the total variation distance. We then rigorous introduce the asymptotic cut-off phenomenon and its relation with the mixing times. In Section~\ref{sub: main results}, we present the main results of the manuscript,  see Theorem~\ref{mainresult} and
Theorem~\ref{mainresultW} there. They give the profile cut-off phenomenon for~\eqref{SDE} in the total variation distance and a renormalized Wasserstein distance, respectively.  
In addition, explicit mixing time asymptotics are presented.
In Section~\ref{proof of mains results}  detailed proofs of Theorem~\ref{mainresult} and
Theorem~\ref{mainresultW} are presented.
More precisely, Theorem~\ref{mainresult} is showed in Subsection~\ref{sub:resulttv}, and 
Theorem~\ref{mainresultW} is proved in 
Subsection~\ref{subsec TW}.
Finally, we provide an appendix with auxiliary results that have been
used throughout the manuscript. The appendix is divided
in five sections. More specifically, 
in Appendix~\ref{apendtv} we briefly recall general basic properties of the total variation distance. Then we introduce the 
Parseval--Plancherel--Fourier approach for total variation convergence.   
In Appendix~\ref{apen:gamma} we show a useful scaling property for the Gamma distribution and prove a local central limit theorem for the stationary probability measure of~\eqref{SDE}. 
In Appendix~\ref{apenddist} we recall the convergence to equilibrium (in the total variation distance and in the Wasserstein distance) for the evolution of~\eqref{SDE}. In addition, we recall explicit formulas for the characteristic function and moment generating function of~\eqref{SDE}. In Appendix~\ref{ap:tools} we present asymptotic expansions that are crucial in the proofs of the main results.  Lastly in Appendix~\ref{mixtime}, mixing times equivalence of asymptotic profile cut-off is presented.

\subsection{\textbf{Preliminaries and asymptotic cut-off phenomenon}}\label{sub:definitions}
In this section, we review the necessary background about Wasserstein distance and the total variation distance. 
We point out that the weak convergence in the space of probability  measures with finite $p$-moment can be metrized by the Wasserstein distance.
In addition, it is a natural way to compare the discrepancy between the laws of two random elements, even for degenerate cases, where one random element is derived from the other by a small perturbation.
Nowadays, it is an important and well-studied mathematical concept in its own right.
Moreover, it has been ubiquitous in applications
 such as 
optimal transport theory, probability theory, partial differential equations, machine learning,
for further details and properties of the Wasserstein distance, we refer to the 
monographs~\cite{Ambrosio,Figalli,Panaretos2020, Peyreeee,Villani2009}.

We now define the asymptotic cut-off phenomenon in a general setting.
In what follows, we recall the definition
of coupling between two probability measures and then define the Wasserstein distance and the total variation distance. Let $(H,d)$ be a Polish space, a complete separable metric space,
equipped with its Borel $\sigma$-algebra 
$\mathcal{B}(H)$ and let $\mathbb{M}_1(H)$ be the set of probability measures defined in the measurable space $(H,\mathcal{B}(H))$.
Given $\nu_1\in \mathbb{M}_1(H)$ and $\nu_2\in \mathbb{M}_1(H)$,
we say that a probability measure $\Pi$ defined in the product measurable space $(H\times H,\mathcal{B}(H)\otimes \mathcal{B}(H))$ is a coupling (a.k.a. transportation plan or joint distribution) between $\nu_1$ and $\nu_2$ if and only if
$\Pi(B\times H)=\nu_1(B)$ and
$\Pi(H\times B)=\nu_2(B)$
for any measurable set $B\in \mathcal{B}(H)$.
We then denote by $\mathcal{C}(\nu_1,\nu_2)$ the set of all couplings between $\nu_1$ and $\nu_2$. We note that the product measure $\nu_1\otimes \nu_2$ is a coupling between $\nu_1$ and $\nu_2$ and hence 
$\mathcal{C}(\nu_1,\nu_2)\not=\emptyset$.  

For any $p>0$ we denote by $\mathbb{M}_{1,p}$ the set of Borel probability measures on $(H,\mathcal{B}(H))$ with 
 finite $p$-th moment, that is, 
$\mathbb{M}_{1,p}\subset \mathbb{M}_1$ and for some $u_0\in H$ it follows that
$\int_{H} d(u,u_0)^p \nu(\ud u)<\infty$ for all $\nu\in \mathbb{M}_{1,p}$ (the choice of the point $u_0$ is irrelevant).
Then the Wasserstein distance of order $p>0$,
$\mathcal{W}_p:\mathbb{M}_{1,p}\times \mathbb{M}_{1,p}\to [0,\infty)$, is defined as
\begin{equation}\label{eq:Wdef}
\begin{split}
\mathcal{W}_p(\nu_1,\nu_2)&:=\inf_{\Pi\in \mathcal{C}(\nu_1,\nu_2)}
\left(\int_{H\times H} d(u,v)^p\Pi(\ud u,\ud v)\right)^{1\wedge (1/p)},
\end{split}
\end{equation}
where for any real numbers $a$ and $b$ we denote the minimum between $a$ and $b$ by $a\wedge b$.
Roughly speaking,~\eqref{eq:Wdef} measures how expensive is the transportation of the measure $\nu_1$ into the measure $\nu_2$ with respect to the (Monge--Kantorovich) cost function $c(u,v):=d(u,v)^p$, $u,v\in H$.
We also define the total variation distance $\tv:\mathbb{M}_{1}\times \mathbb{M}_{1}\to [0,1]$ by
\begin{equation}\label{eq:tvdef}
\tv(\nu_1,\nu_2):=
\inf_{\Pi\in \mathcal{C}(\nu_1,\nu_2)}
\int_{H\times H} \mathbf{1}_{\{u\neq v\}}\Pi(\ud u,\ud v),
\end{equation}
where $\mathbf{1}_{U}$ denotes the indicator function of a given set $U\subset H$.

Let $X_1$ and $X_2$ be two random elements  defined on the probability space $(\Omega,\mathcal{F},\mathbb{P})$ with finite $p$-th moment and taking values on $H$.
The Wasserstein distance of order $p>0$ between $X_1$ and $X_2$
is defined by 
$\mathfrak{W}_p(X_1,X_2):=\mathcal{W}_p(\mathbb{P}_{X_1},\mathbb{P}_{X_2})$, where 
$\mathbb{P}_{X_1}$ and $\mathbb{P}_{X_2}$ are the push-forward
probability measures
$\mathbb{P}_{X_1}(B):=\mathbb{P}(X_1\in B)$ and 
$\mathbb{P}_{X_2}(B):=\mathbb{P}(X_2\in B)$ 
for any $B\in \mathcal{B}(H)$.
For convenience, 
we write $\mathcal{W}_p(X_1,X_2)$ in place of $\mathfrak{W}_p(X_1,X_2)$.
Similarly, we write 
$\tv(X_1,X_2)$ instead of $\tv(\mathbb{P}_{X_1},\mathbb{P}_{X_2})$.
Moreover, for the sake of intuitive reasoning and in a conscious abuse of notation we write
$\mathcal{W}_p(X_1,\nu_2)$
and 
$\tv(X_1,\nu_2)$ instead of $\mathcal{W}_p(X_1,X_2)$
and 
$\tv(X_1,X_2)$, respectively, where $\nu_2:=\mathbb{P}_{X_2}$. In other words,
\begin{equation}
\begin{split}
\mathcal{W}_p(\nu_1,\nu_2)&=\inf_{(X_1,X_2)}
\left(\mathbb{E}_{\mathbb{P}}[d(X_1,X_2)^p]\right)^{1\wedge (1/p)}
\end{split}
\end{equation}
and
\begin{equation}
\begin{split}
\tv(\nu_1,\nu_2)=
\inf_{(X_1,X_2)}
\mathbb{P}(X_1\neq X_2),
\end{split}
\end{equation}
where the infimum is taken over all pairs $(X_1,X_2)$ satisfying $\mathbb{P}_{X_j}=\nu_j$, $j=1,2$ and $\mathbb{E}_{\mathbb{P}}$ denotes the expectation with respect to $\mathbb{P}$.
While by definition, the Wasserstein distance or the total variation distance are minimizers of an expected value of a given cost function, they are always bounded
above by evaluating the expectation of the cost function in any coupling (for instance, the synchronous coupling), however, lower bounds are typically hard to establish.

In the sequel,  we embed our results in the context of cut-off phenomenon in a general framework.
The complexity parameter of the model is denoted by $\e>0$.
Let $(H_\e,d_\e)$ be a Polish space equipped with its corresponding Borel $\sigma$-algebra $\mathcal{B}(H_\e)$ and let $\mathbb{M}_1(H_\e)$ be the set of probability measures defined on $\mathcal{B}(H_\e)$.
Let $\widetilde{\mathbb{M}}_1(H_\e)\subset \mathbb{M}_1(H_\e)$ be
equipped with a distance or divergence $\mathrm{dist}_\e:\widetilde{\mathbb{M}}_1(H_\e)\times \widetilde{\mathbb{M}}_1(H_\e)\to [0,\infty]$,
and set
\[
\mathrm{Diam}_\e:=\sup_{\nu_1,\nu_2\in \widetilde{\mathbb{M}}_1(H_\e)}\mathrm{dist}_\e(\nu_1,\nu_2)\]
and assume that 
\begin{equation}\label{diameter}
\lim_{\e\to 0^+}\, \mathrm{Diam}_\e=:\mathrm{Diam}\in (0,\infty].
\end{equation}
We then consider $\mathcal{X}^{\e,x_\e}:=(\mathcal{X}^{\e}_t(x_\e))_{t\gqq 0}$ a random dynamical system (including deterministic dynamical system) in continuous time and taking values in $H_\e$, where $x_\e\in H_\e$ denotes the initial datum of the process $\mathcal{X}^{\e,x_\e}$.
In addition, we assume that for any $t\gqq 0$ the law of
$\mathcal{X}^{\e}_t(x_\e)$, here denoted by $\mathrm{Law}\,\mathcal{X}^{\e}_t(x_\e)$, belongs to $\widetilde{\mathbb{M}}_1(H_\e)$ and the existence of $\nu^\e\in \widetilde{\mathbb{M}}_1(H_\e)$ such that
\[\lim\limits_{t\to \infty}\mathrm{dist}_\e(
\mathrm{Law}\,\mathcal{X}^{\e}_t(x_\e),\nu^\e)=0.
\]
Since we have already fixed the distance $\mathrm{dist}_\e$ for the convergence to the limiting law $\nu^{\e}$,
we now define the asymptotic cut-off phenomenon.

\begin{definition}[Asymptotic cut-off phenomenon]\label{def:asymst}
We say that the family of random dynamical systems $(\mathcal{X}^{\e,x_\e},\nu^\e,\mathrm{dist}_\e)_{\e>0}$ exhibits:
\begin{itemize}
\item[(i)] 
\textit{asymptotic pre-cut-off}  
at  $(t_{\e,x_\e})_{\e>0}$ with pre-cut-off time $t_\e:=t_{\e,x_\e}$
if
$t_{\e}\to \infty$, as $\e \to 0^+$
and there are constants $0<C_*<C^*<\infty$ such that  
\begin{equation}\label{def:precut1}
\lim\limits_{\e \rightarrow 0^+}\,
\mathrm{dist}_\e(\mathrm{Law}\,\mathcal{X}^{\e}_{\delta\cdot t_{\e}}(x_\e),\nu^\e)=
\begin{cases}
\mathrm{Diam} & \quad 
\textrm{ for } \quad \delta=C_*,\\
0 & \quad \textrm{ for } \quad\delta=C^*.
\end{cases}
\end{equation}
\item[(ii)] \textit{asymptotic cut-off}  
at $(t_{\e,x_\e})_{\e>0}$ with cut-off time  $t_\e:=t_{\e,x_\e}$
if
$t_{\e}\to \infty$, as $\e \to 0^+$ and  
\begin{equation}\label{def:cut1}
\lim\limits_{\e \rightarrow 0^+}\,
\mathrm{dist}_\e(\mathrm{Law}\,\mathcal{X}^{\e}_{\delta\cdot t_{\e}}(x_\e),\nu^\e)=
\begin{cases}
\mathrm{Diam} & \quad \textrm{ for } \quad \delta\in(0,1),\\
0 & \quad \textrm{ for } \quad\delta\in (1,\infty).
\end{cases}
\end{equation}
\item[(iii)] 
\textit{asymptotic window cut-off}  at $((t_{\e,x_\e},\omega_{\e,x_\e}))_{\e>0}$
with cut-off time $t_\e:=t_{\e,x_\e}>0$ 
and time window $\omega_\e:=\omega_{\e,x_\e}>0$
if
$t_{\e}\to \infty$, as $\e \to 0^+$,
$\lim\limits_{\e\to 0^+}\frac{\omega_{\e}}{t_{\e}}=0$, and the limits
\begin{equation}\label{def:cut2}
G_*(r):=\liminf\limits_{\e \rightarrow 0^+}\,
\mathrm{dist}_\e(\mathrm{Law}\,\mathcal{X}^{\e}_{t_{\e}+r\cdot \omega_{\e}},\nu^\e)
\end{equation}
and
\begin{equation}\label{def:cut2nuevo}
G^*(r):=\limsup\limits_{
\e \rightarrow 0^+}\,\mathrm{dist}_\e(\mathrm{Law}\,\mathcal{X}^{\e}_{t_{\e}+r\cdot \omega_{\e}},\nu^\e)
\end{equation}
satisfy
\[
\lim\limits_{r\rightarrow -\infty}G_*(r)=\mathrm{Diam}\quad 
\textrm{ and }
\quad
\lim\limits_{r\rightarrow \infty}G^*(r)=0.
\]
\item[(iv)] \textit{asymptotic profile cut-off} at 
$((t_{\e,x_\e},\omega_{\e,x_\e}))_{\e>0}$ 
with cut-off time $t_\e:=t_{\e,x_\e}>0$, 
time window $\omega_\e:=\omega_{\e,x_\e}>0$ and
profile function $G:\mathbb{R}\rightarrow [0,\mathrm{Diam}]$
if 
$t_{\e}\to \infty$, as $\e \to 0^+$,
$\lim\limits_{\e\to 0^+}\frac{\omega_{\e}}{t_{\e}}=0$,
\begin{equation}\label{def:cut3}
\lim\limits_{\e \rightarrow 0^+}\,\mathrm{dist}_\e(\mathrm{Law}\,\mathcal{X}^{\e}_{t_{\e}+r\cdot \omega_{\e}},\nu^\e)=:G(r) \quad 
\textrm{ exists for any }\quad r\in \mathbb{R}.
\end{equation}
In addition, it satisfies
\[
\lim\limits_{r\rightarrow -\infty}\,G(r)=\mathrm{Diam}\quad 
\textrm{ and } 
\quad\lim\limits_{r\rightarrow \infty}\,G(r)=0.
\]
\end{itemize}
\end{definition}

We point out that find an explicit profile function $G$ is in general difficult.
Although it is believed that many families of processes exhibit asymptotic cut-off phenomenon, showing such occurrence is often an extremely challenging task 
even for simple Markov chains. It requires the full understanding of the behavior of the underlying dynamics.
It is clear that (iv) implies (iii) and (ii) implies (i). 

In what follows, we assume that for any fixed $\e>0$, the map 
\begin{equation}\label{eq:monotonic}
[0,\infty)\ni t \mapsto  \mathrm{dist}_\e(\mathrm{Law}\,\mathcal{X}^{\e}_{t},\nu^\e)\quad 
\textrm{ is monotonic decreasing.}   
\end{equation}
By~\eqref{eq:monotonic} we have that $G_*$, $G^*$  are monotonic decreasing.
We note that~\eqref{eq:monotonic} is natural in the context of homogeneous Markov processes and holds true for the so-called Dyson--Ornstein--Uhlenbeck process and
various distances/discrepancies of interest, see for instance Lemma~B.3 (Monotonicity) in~\cite{Boursier}.
As it is remarked in the ``Elements of proof of Lemma B.3" the monotonicity comes from the Markov nature of the process and the so-called $\Phi$-convexity, see p.~507 in~\cite{Boursier} for further details.

Under~\eqref{eq:monotonic} one can see that (iii) implies~(ii).

In the sequel, we define the so-called $\eta$-mixing time.
For any $\eta\in (0,\mathrm{Diam}_\e)$ the $\eta$-mixing time, $\tau^{\e,\mathrm{dist}_{\e}}_{\mathrm{mix}}(\eta)$, is defined by
\begin{equation}\label{def:mixingtimes}
\tau^{\e,\mathrm{dist}_{\e}}_{\mathrm{mix}}(\eta):=\inf\{t\gqq 0: \mathrm{dist}_\e(\mathrm{Law}\,\mathcal{X}^{\e}_{t},\nu^\e)\lqq \eta\}.  
\end{equation}
In other words, $\tau^{\e,\mathrm{dist}_{\e}}_{\mathrm{mix}}(\eta)$ is the time required by the process $\mathcal{X}^\e$
to be close to its limiting law up to a prescribed error $\eta\in (0,\mathrm{Diam}_\e)$.
For discrete finite state space, the definition of $\eta$-mixing time additionally has an optimization over the worst possible initial datum of the random dynamics.
However, since our purpose is to describe
the dependence of the mixing time with respect to the initial state of the process, we use the definition~\eqref{def:mixingtimes} for $\eta$-mixing time.
We also stress that the cut-off times may have strong dependence with respect to the initial datum of the random dynamics,
see for instance~\cite{BJJSP,BJAAP}.

For the total variation distance, it is known that the asymptotic
cut-off phenomenon can be also interpreted as a mixing time, see Chapter~18 of~\cite{Levin2009} or~\cite{BarreraJaviera1}. 

In the sequel, we show that the asymptotic cut-off phenomenon in the preceding framework (Definition~\ref{def:asymst}) can be also interpreted as a mixing time in the following precise sense.

\begin{proposition}[Asymptotics for the mixing times]\label{proposition asymptotic}
The following statements are valid.
\begin{itemize}
\item[(1)]
Assume that (i)  holds true. 
There exist constants $0<c_*<c^*<\infty$ such that
for any $\eta\in (0,\mathrm{Diam})$ it follows that
\begin{equation}
c_*<\liminf\limits_{\e\to 0^+}
\frac{t_\e}
{\tau^{\e,\mathrm{dist}_{\e}}_{\mathrm{mix}}(\eta)}
\lqq 
\limsup\limits_{\e\to 0^+}
\frac{t_\e}
{\tau^{\e,\mathrm{dist}_{\e}}_{\mathrm{mix}}(\eta)}<c^*.
\end{equation}

\item[(2)]
Assume that (ii) holds true. Then for any $\eta\in (0,\mathrm{Diam})$ it follows that
\begin{equation}
\lim\limits_{\e\to 0^+}
\frac{t_\e}
{\tau^{\e,\mathrm{dist}_{\e}}_{\mathrm{mix}}(\eta)}=1.
\end{equation}

\item[(3)]
Assume that (iii) holds true. 
In addition, assume that $G^*$ and $G_*$ are continuous and strictly decreasing.
Then for any $\eta\in (0,\mathrm{Diam})$ it follows that
\begin{equation}
 t_\e+\omega_\e \cdot (G_{*})^{-1}(\eta)+\mathrm{o}(\omega_\e)\lqq 
\tau^{\e,\mathrm{dist}_{\e}}_{\mathrm{mix}}(\eta)\lqq t_\e+\omega_\e \cdot (G^{*})^{-1}(\eta)+\mathrm{o}(\omega_\e),
\end{equation}
where $(G_*)^{-1}(\eta):=\inf\{r\gqq 0: G_*(r)\lqq \eta\}$
 and 
 $(G^*)^{-1}(\eta):=\inf\{r\gqq 0: G^*(r)\lqq \eta\}$.

\item[(4)]
Assume that (iv) holds true. 
In addition, assume that $G$ is continuous and strictly decreasing.
Then for any $\eta\in (0,\mathrm{Diam})$ it follows that
\begin{equation}
\tau^{\e,\mathrm{dist}_{\e}}_{\mathrm{mix}}(\eta)=t_\e+ G^{-1}(\eta)\cdot \omega_\e +\mathrm{o}(\omega_\e),\quad 
\textrm{ as }\quad \e\to 0^+,
\end{equation}
where $G^{-1}(\eta):=\inf\{r\in \mathbb{R}: G(r)\lqq \eta\}$ and $\mathrm{o}(\omega_\e)/\omega_\e\to 0^+$ as $\e\to 0^+$.
\end{itemize}
\end{proposition}
The proof is given in 
Appendix~\ref{mixtime}.

\begin{remark}[Non-uniqueness of cut-off times and time windows]
Assume that~\eqref{eq:monotonic} is valid.
In addition, assume that the functions $G_*$, $G^*$ given in
Item~(iii) of Definition~\ref{def:asymst} are continuous.
We point out that the cut-off times and time windows are unique up to an  equivalent relation.
More precisely,
if 
$((t'_{\e,x_\e},\omega'_{\e,x_\e}))_{\e>0}$ are functions
such that 
\[
\lim _{\e\to 0^{+}} \frac{t_{\e,x_\e}}{ t'_{\e,x_\e}}=1
\quad \textrm{ and }\quad
\lim_{\e\to 0^{+}} \frac{\omega_{\e,x_\e}}{ \omega'_{\e,x_\e}}=c>0,
\]
we have that 
Item~(iii) in Definition~\ref{def:asymst} hold true for 
$((t'_{\e,x_\e},\omega'_{\e,x_\e}))_{\e>0}$ and the functions $G_*$ and $G^*$ are now evaluated at $r/c$.
In particular, when Item~(iv) in Definition~\ref{def:asymst} hold true and $G$ is continuous, Item~(iv) in Definition~\ref{def:asymst} hold true for 
$((t'_{\e,x_\e},\omega'_{\e,x_\e}))_{\e>0}$ and the function $G$ is evaluated at $r/c$.
\end{remark}

We remark that the previous definition of asymptotic cut-off phenomenon naturally applies for discrete time systems using the floor function in the running time of the system, that is, 
replacing $t$ by $\lfloor t\rfloor$, accordingly.
In the context of ergodic Markov chains with finite state space, the asymptotic cut-off phenomenon is typically measure in the total variation distance and the complexity parameter is the cardinality of the state space. In addition, it is shown that it can be expressed at increasingly sharp levels, i.e., (iv) implies (iii), (iii) yields (ii) and (ii) gives (i), but the reciprocal implications are not valid in general.

Bearing all this in mind
we provide a complete characterization of when cut-off occurs which is exactly the content  of the following subsection.

\subsection{\textbf{Main results and their consequences}}
\label{sub: main results}
In this section, we establish the main results of this manuscript. 
Recall that for any fixed noise intensity $\e\in (0,\sqrt{2\fb})$ there exists a unique invariant probability measure $\mu^\e$ for the random dynamics given 
in~\eqref{SDE}, that is, $X^\e_t(\mu^\e)\stackrel{d}{=}\mu^\e$ for any $t\gqq 0$, where $\stackrel{d}{=}$ denotes equality in distribution sense.
In addition, for any $p>0$
and any initial condition $x\gqq 0$ of~\eqref{SDE} it follows that
\begin{equation}
\lim\limits_{t\to \infty}\mathcal{W}_p
\left(X^{\e}_t(x),\mu^\e\right)=0\quad \textrm{ and } \quad 
\lim\limits_{t\to \infty}\tv
\left(X^{\e}_t(x),\mu^\e\right)=0,
\end{equation}
see Lemma~\ref{lem:ergodicity} in Appendix~\ref{apenddist}.
We note that $\mathrm{Diam}_{\mathcal{W}_p}=\infty$ and $\mathrm{Diam}_{\mathrm{TV}}=1$.

Following the notation in Subsection~\ref{sub:definitions}, 
the complexity parameter is $\e\in (0,\sqrt{2\fb})$, 
$H_\e:=\mathbb{R}$ equipped with the standard Euclidean distance, 
$x_\e:=x\in [0,\infty)$,
$\mathcal{X}^{\e,x_\e}:=X^{\e,x_\e}$ is given in~\eqref{SDE} and $\nu^\e:=\mu^\e$ is the invariant probability measure for~\eqref{SDE}.

The first main result of this manuscript is the following profile cut-off phenomenon in total variation distance, i.e., $\widetilde{\mathbb{M}}_1(\mathbb{R}):=\mathbb{M}_1(\mathbb{R})$ and 
$\mathrm{dist}_\e:=\tv$.
 
\begin{theorem}[Asymptotic profile cut-off phenomenon for CIR models I]\label{mainresult}
\hfill
Assume that $x\in [0,\infty)\setminus\{\frac{\fb}{\fa}\}$.
The family of CIR models 
$(X^{\e,x})_{\e>0}$ defined 
in~\eqref{SDE} exhibits asymptotic profile cut-off  phenomenon as $\e$ tends to zero in the total variation distance
\begin{equation}\label{eq:timecut}
\textrm{at cut-off time }\quad
t_{\e}:=
\frac{1}{\fa}\ln\left(\frac{1}{\e}\right)\quad \textrm{and time window }\quad \omega_{\e}:=\frac{1}{\fa}.
\end{equation}
In other words, for any $r\in \mathbb{R}$ the following limit is valid
\begin{equation}\label{eq:Gx}
\lim\limits_{\e\to 0^+}
\tv\left(X^{\e}_{t_\e+r\cdot \omega_\e}(x),\mu^\e\right)=
\tv\left(|C_x|\cdot e^{-r}+\mathcal{G},\mathcal{G}\right)=:G^{\mathrm{TV}}_x(r),
\end{equation}
where the constant $C_x$ is given by
\begin{equation}\label{def:Cx}
C_x:=\frac{\sqrt{2\fb}}{\fb}(\fa x-\fb)=\frac{\fa\sqrt{2\fb}}{\fb}(x-\fb/\fa),
\end{equation}
and $\mathcal{G}$ denotes a random variable with standard Gaussian law. 
Moreover, 
\begin{equation}\label{eq:inftylimites}
\lim_{r\to -\infty} G^{\mathrm{TV}}_x(r)=1\quad 
\textrm{ and }\quad  \lim_{r\to \infty} G^{\mathrm{TV}}_x(r)=0.
\end{equation}
In addition, for $x=\frac{\fb}{\fa}$
and for any function $  (s_\e)_{\e>0}$ such that $s_\e\to \infty$ as $\e\to 0^+$
it follows that
\begin{equation}\label{eq:nocutf}
\lim\limits_{\e\to 0^+}
\tv\left(X^{\e}_{s_{\e}}(x),\mu^\e\right)
=0.
\end{equation}
In particular, the family of CIR models 
$(X^{\e,\fb/\fa})_{\e>0}$ defined in~\eqref{SDE}  does not exhibit  asymptotic cut-off phenomenon as $\e$ tends to zero in the total variation distance.
\end{theorem}
The proof of Theorem~\ref{mainresult} is presented in Subsection~\ref{sub:resulttv}. 
For clarity of the presentation,~\eqref{eq:Gx} is proved in Lemma~\ref{lem:profile} with the help of Lemma~\ref{lem:lcltx}, Lemma~\ref{lem:replacement} and Lemma~\ref{lem:limites}.
While~\eqref{eq:nocutf} is shown in Lemma~\ref{lem:nocutoff}.
Finally, the proof~\eqref{eq:inftylimites} is given in Remark~\ref{re:stvg}. 

For shorthand and in a conscious abuse of notation, we use indistinctly the following notations for the exponential function: $\exp(a)$ or $e^a$ for $a\in \mathbb{R}$.
In addition, we use $|\cdot|$ for the absolute value function and for the modulus of a complex number.

We continue to rely on the notations and assumptions in Theorem~\ref{mainresult}.

\begin{remark}[Asymmetry for the profile function]
We observe that the function 
\begin{equation}\label{def:Cxjh}
[0,2\fb/\fa]\ni x\mapsto |C_x|=\frac{\fa\sqrt{2\fb}}{\fb}|x-\fb/\fa|
\end{equation}
is symmetric around $\fb/\fa$, while for $x>2\fb/\fa$ there is a clearly asymmetry, that is, $|C_x|>|C_{x'}|$ for any $x'\in [0,2\fb/\fa]$. This implies an asymmetry of $G^{\mathrm{TV}}_x$ with respect to the initial datum $x\in [0,\infty)$.
\end{remark}

\begin{remark}[A word about non-cutoff phenomenon in~\eqref{eq:nocutf}]\label{rem:nocut}
For any initial datum $x\gqq 0$,
the variation of constants formula with the help of It\^o's Lemma yields that
the path-wise solution of~\ref{SDE} can be represented as
\begin{equation}\label{eq:duhamel}
\begin{split}
X^{\e}_t(x)
&=e^{-\fa t}\left(x-\frac{\fb}{\fa}\right)+\frac{\fb}{\fa}+\e\mathcal{O}^{\e}_t(x)
\end{split}
\end{equation}
for all $t\gqq 0$, where
\[
\mathcal{O}^{\e}_t(x):=e^{-\fa t}\int_{0}^{t}e^{\fa s}\sqrt{X^{\e}_s(x)}\ud B_s,\quad t\gqq 0
\]
By ergodicity, for any initial condition $x\gqq 0$ we have that $\mathcal{O}^{\e}_t(x)$ converges in distribution as $t\to \infty$ to a random variable $\mathcal{O}^{\e}_\infty$. We point out that the distribution of $\mathcal{O}^{\e}_\infty$ does not depend on $x$.
By Lemma~\ref{lem:lcltx} we have
\begin{equation}\label{eq:aca}
\lim\limits_{\e\to 0^+}\tv\left(\frac{\mathcal{O}^{\e}_\infty}{\sigma}, \mathcal{G}\right)=0,
\end{equation}
where $\sigma:=\frac{\sqrt{2\fb}}{2\fa}$.
For $x=\fb/\fa$,
Lemma~\ref{lem:lcltx} and Lemma~\ref{lem:nocutoff}
with the help of Lemma~\ref{lem:trasca} in Appendix~\ref{apendtv}
imply
for any function $(s_\e)_{\e>0}$ satisfying $s_\e\to \infty$ as $\e\to 0^+$ that
\begin{equation}\label{eq:zeroli}
\lim\limits_{\e\to 0^{+}}\tv(\mathcal{O}^{\e}_{s_\e}(x), \sigma\mathcal{G})=0,
\end{equation}
where $\mathcal{G}$ denotes the standard Gaussian distribution. 
Combining~\eqref{eq:aca} and~\eqref{eq:zeroli} in~\eqref{eq:duhamel} for $x=\fb/\fa$ we have 
\[
X^{\e}_{s_\e}(x)\approx \frac{\fb}{\fa}+\e \sigma \mathcal{G}\quad \textrm{ and }\quad  X^{\e}_{\infty}\approx \frac{\fb}{\fa}+\e \sigma \mathcal{G}, \quad 
\textrm{ as }\quad \e\to 0^+,
\]
and therefore there is no cut-off phenomenon 
as $\e$ tends to zero in the total variation distance for the initial datum $x=\fb/\fa$.
\end{remark}

\begin{remark}[Shape of the asymptotic profile function in the total variation distance]\label{re:stvg}
For any $m\in \mathbb{R}$ it follows that
\[
\tv\left(m+\mathcal{G},\mathcal{G}\right)=\frac{2}{\sqrt{2\pi}}\int\limits_{0}^{|m|/2} \exp\left(-\frac{z^2}{2}\right) \ud z,
\]
for details we refer to Item~(i) of Lemma~B.2 in Appendix~B 
of~\cite{BarreraLiu2021}.
Hence, for $G^{\mathrm{TV}}_x$ defined in~\eqref{eq:Gx} it follows that
\[
G^{\mathrm{TV}}_x(r)=\frac{2}{\sqrt{2\pi}}\int\limits_{0}^{|C_x|\cdot  e^{-r}/2} \exp\left(-\frac{z^2}{2}\right) \ud z\quad \textrm{ for any }\quad r \in \mathbb{R},
\]
which with the help of the Monotone Convergence Theorem implies
\[\lim\limits_{r\to -\infty}G^{\mathrm{TV}}_x(r)=1
\quad
\textrm{ and }
\quad\lim\limits_{r\to \infty}G^{\mathrm{TV}}_x(r)=0.
\]
In addition, the Fundamental Theorem of Calculus and the Mill ratio give asymptotics for the profile function $G^{\mathrm{TV}}_x$. To be more precise,
the asymptotically exponential shape of the profile for $r \gg 1$, that is,
\[
G^{\mathrm{TV}}_x(r)\stackrel{r\to \infty}{\sim} \frac{1}{\sqrt{2\pi}}|C_x|\cdot e^{-r},
\]
whereas, the asymptotically doubly exponential shape of the profile for $r \ll -1$
\[
1-G^{\mathrm{TV}}_x(r)\stackrel{r\to -\infty}{\sim} \frac{4}{\sqrt{2\pi}}\frac{1}{|C_x|\cdot e^{-r}}\exp\left(-\frac{C^2_x \cdot e^{-2r}}{8}\right),
\]
where for short we write $F_1 \stackrel{r\to \infty}{\sim} F_2$ and 
$F_1 \stackrel{r\to -\infty}{\sim} F_2$
in a place of $\lim\limits_{r \to \infty} \frac{F_1(r)}{F_2(r)} = 1$ and
$\lim\limits_{r \to -\infty} \frac{F_1(r)}{F_2(r)} = 1$, respectively.
\end{remark}

As a consequence of Theorem~\ref{mainresult} with the help of Proposition~\ref{proposition asymptotic} and Remark~\ref{re:stvg} we obtain the following corollary, which provides the asymptotic behavior for the mixing times.
\begin{corollary}[$\eta$-mixing time in the total variation distance]\label{cor:tvmix}
For any $\eta>0$ it follows that
\begin{equation}
\tau^{\e,\tv}_{\mathrm{mix}}(\eta)=t_{\e}+\omega_{\e}\cdot \left(G^{\mathrm{TV}}_x\right)^{-1}(\eta)+\mathrm{o}_{\e\to 0^+}(1),
\end{equation}
where $\left(G^{\mathrm{TV}}_x\right)^{-1}(\eta):=\inf\{r\gqq 0:G^{\mathrm{TV}}_x(r)\lqq \eta \}$.
\end{corollary}

\begin{remark}[The mixing time at $\eta=1/4$]
For $\eta=\frac{1}{4}$, we set 
$\tau^{\e,\tv}_{\mathrm{mix}}:=\tau^{\e,\tv}_{\mathrm{mix}}(\eta)$ and then
Corollary~\ref{cor:tvmix} yields that
\begin{equation}
\tau^{\e,\tv}_{\mathrm{mix}}=t_{\e}+\omega_{\e}\cdot \left(G^{\mathrm{TV}}_x\right)^{-1}(1/4)+\mathrm{o}_{\e\to 0^+}(1).
\end{equation}
By Remark~\ref{re:stvg} we have that 
\[
G^{\mathrm{TV}}_x(r)=\frac{2}{\sqrt{2\pi}}\int\limits_{0}^{|C_x|\cdot  e^{-r}/2} \exp\left(-\frac{z^2}{2}\right) \ud z\quad \textrm{ for any }\quad r \in \mathbb{R}.
\]
Now, for simplicity choose the constants $\fa>0$, $\fb>0$ and $x\gqq 0$ such that
$C_x=2$ and then find the unique $r>0$ such that
\[
\int\limits_{0}^{  e^{-r}} \exp\left(-\frac{z^2}{2}\right) \ud z=\frac{\sqrt{2\pi}}{8}\approx 0.3133,\quad 
\textrm{ that is, }\quad r\approx 1.1435.
\]
In summary, we obtain
\begin{equation}
\tau^{\e,\tv}_{\mathrm{mix}}\approx \frac{1}{\fa}\ln\left(\frac{1}{\e}\right)+\frac{1}{\fa}\times  1.1435 \quad \textrm{ for }\quad \e\ll 1.
\end{equation}
\end{remark}

The second main result of this paper is the following profile cut-off phenomenon in a renormalized Wasserstein distance of order $p>0$, 
i.e., 
$\widetilde{\mathbb{M}}_1(\mathbb{R}):=\mathbb{M}_{1,p}(\mathbb{R})$ and
$\mathrm{dist}_\e:=\frac{1}{\e^{1\wedge p}}\mathcal{W}_p$.
\begin{theorem}[Asymptotic profile cut-off  phenomenon for CIR models II]\label{mainresultW}
\hfill
Assume that $x\in [0,\infty)\setminus\{\frac{\fb}{\fa}\}$.
The family of CIR models 
$(X^{\e,x})_{\e>0}$ defined 
in~\eqref{SDE} exhibits asymptotic profile cut-off  phenomenon as $\e$ tends to zero in the renormalized Wasserstein distance of order $p>0$, $\e^{-(1\wedge p)}\mathcal{W}_p$, at cut-off time $t_{\e}$ and time window $\omega_{\e}$ defined 
in~\eqref{eq:timecut}.
In other words, for any $r\in \mathbb{R}$ the following limit holds true
\begin{equation}\label{eq:GxW}
\lim\limits_{\e\to 0^+}
\frac{\mathcal{W}_p\left({X^{\e}_{t_\e+r\cdot \omega_\e}(x),\mu^{\e}}\right)}{\e^{1\wedge p}}
=\left(\frac{\sqrt{2\fb}}{2\fa}\right)^{1\wedge p}
\mathcal{W}_p\left(|C_x|\cdot e^{-r}+\mathcal{G},\mathcal{G}\right)
:=G^{\mathcal{W}_p}_{x}(r),
\end{equation}
where the constant $C_x$ is given in~\eqref{def:Cx} and $\mathcal{G}$ denotes a random variable with standard Gaussian law. 
Moreover, 
\begin{equation}\label{eq:inftylimitesW}
\lim_{r\to -\infty} G^{\mathcal{W}_p}_{x}(r)=\infty\quad \textrm{ and }\quad 
\lim_{r\to \infty} G^{\mathcal{W}_p}_{x}(r)=0.
\end{equation}

\noindent
In addition, for $x=\frac{\fb}{\fa}$ and for any function $(s_\e)_{\e>0}$
such that $s_\e\to \infty$ as $\e\to 0^+$
it follows that
\begin{equation}\label{eq:wanocutf}
\lim\limits_{\e\to 0^+}\frac{\mathcal{W}_p\left({X^{\e}_{s_\e}(x),\mu^{\e}}\right)}{\e^{1\wedge p}}=0.
\end{equation}
In particular, the family of CIR models 
$(X^{\e,\fb/\fa})_{\e>0}$ defined in~\eqref{SDE}  does not exhibit  asymptotic cut-off phenomenon as $\e$ tends to zero in the renormalized Wasserstein distance.
\end{theorem}

The proof of Theorem~\ref{mainresultW} is presented in Subsection~\ref{subsec TW}. 
For clarity of the presentation,~\eqref{eq:GxW} is proved in Lemma~\ref{lem:profileWA} with the help of Lemma~\ref{lem:lcltxW}, Lemma~\ref{lem:reemplazoKRW} and Lemma~\ref{lem:limitesKWR}.
While~\eqref{eq:wanocutf} is shown in Lemma~\ref{lem:nocutoffWA}.
Finally, the proof~\eqref{eq:inftylimitesW} is given in Remark~\ref{rem:shapeWa}.

We continue to rely on the notations and assumptions in Theorem~\ref{mainresultW}.

Analogous observations given in Remark~\ref{rem:nocut} are valid for the renormalized  Wasserstein distance of order $p>0$.

\begin{remark}[Shape of the asymptotic profile function in the Wasserstein distance of order $p>0$]\label{rem:shapeWa}
Let $p\gqq 1$ be fixed.
For any random variable $X$ with finite $p$-th absolute moment and any deterministic number $m\in \mathbb{R}$ the following shift-linearity
\[
\mathcal{W}_p\left(m+X,X\right)=|m|,
\]
holds true, for details we refer to~(2.6) in Item~(d) of Lemma~2.2 
in~\cite{BHPWA}.
Hence, for $G^{\mathcal{W}_p}_{x}(r)$ defined in~\eqref{eq:Gx} it follows that
\[
G^{\mathcal{W}_p}_{x}(r)=\left(\frac{\sqrt{2\fb}}{2\fa}\right)|C_x|\cdot e^{-r}\quad 
\textrm{ for any }\quad r \in \mathbb{R},
\]
which implies
$\lim\limits_{r\to -\infty}G^{\mathcal{W}_p}_{x}(r)=\infty$ for $x\neq \frac{\fb}{\fa}$, and 
$\lim\limits_{r\to \infty}G^{\mathcal{W}_p}_{x}(r)=0$.
While for $p\in (0,1)$ we have 
\[
\left(\frac{\sqrt{2\fb}}{2\fa}\right)^{p}\max\{|C_x|^p\cdot e^{-r\cdot p}-2\mathbb{E}[|\mathcal{G}|^p],0\}
\lqq G^{\mathcal{W}_p}_{x}(r)\lqq \left(\frac{\sqrt{2\fb}}{2\fa}\right)^{p}|C_x|^p\cdot e^{-r\cdot p}
\]
for any $r \in \mathbb{R}$,
see~(2.7) in Item~(d) of Lemma~2.2 
in~\cite{BHPWA}.
Then we have
\[
G^{\mathcal{W}_p}_{x}(r)\stackrel{r\to -\infty}{\sim} \left(\frac{\sqrt{2\fb}}{2\fa}\right)^{p}|C_x|^p\cdot e^{-r\cdot p},
\]
whereas
\[
 G^{\mathcal{W}_p}_{x}(r)\stackrel{r\to \infty}{=}\mathrm{O}\left(\left(\frac{\sqrt{2\fb}}{2\fa}\right)^{p}|C_x|^p\cdot e^{-r\cdot p}\right),
\]
where $\mathrm{O}$ denotes the classical Bachmann--Landau notation.
\end{remark}

As a consequence of Theorem~\ref{mainresultW} with the help of Proposition~\ref{proposition asymptotic} and Remark~\ref{rem:shapeWa} we obtain the following corollary, which provides the asymptotic behavior for the mixing times.

\begin{corollary}[$\eta$-mixing time in the normalized Wasserstein distance]\label{cor:mxw}
For any $\eta>0$ and $p\gqq 1$ it follows that
\begin{equation}
\tau^{\e,\mathcal{W}_p/\e}_{\mathrm{mix}}(\eta)=t_{\e}+\omega_{\e}\cdot \left(G^{\mathcal{W}_p}_x\right)^{-1}(\eta)+\mathrm{o}_{\e\to 0^+}(1),
\end{equation}
where 
\[
\big(G^{\mathcal{W}_p}_x\big)^{-1}(\eta)=\ln\left(\left(\frac{\sqrt{2\fb}}{2\fa}\right)\frac{|C_x|}{\eta}\right)
\]
and the constant $C_x$ is given in~\eqref{def:Cx}.
\end{corollary}

\section{\textbf{Proofs of the main results: Theorem~\ref{mainresult} and Theorem~\ref{mainresultW}}}\label{proof of mains results}

In this section, we give the proof of Theorem~\ref{mainresult} and Theorem~\ref{mainresultW}.  
Recall that $x\gqq 0$, $\fa>0$, $\fb>0$ and $\e\in (0,\sqrt{2\fb})$.
We set
\begin{equation}\label{eq:qepsx0}
\begin{split}
q_\e+1:=\frac{2\fb}{\e^2},\quad c_\e(t):=\frac{2\fa}{\e^2}\frac{1}{1-e^{-\fa t}}
\quad \textrm{ for any }\quad t>0 
\quad \textrm{ and }\quad
c_\e(\infty):=\frac{2\fa}{\e^2}.
\end{split}
\end{equation}
For any positive numbers $r$ and $\theta$ we denote the Gamma distribution with parameter $r$ and $\theta$ by $\Gamma(r,\theta)$. In other words, a random variable $Z$ has distribution $\Gamma(r,\theta)$  if and only if its characteristic function is given by 
\begin{equation}\label{eq:deffourier}
\mathbb{R}\ni u\mapsto \mathbb{E}[e^{\ii u Z}]=\left(1-\ii \theta^{-1}u \right)^{-r}.
\end{equation}
For convenience and in a conscious abuse of notation, we write $\mathbb{E}[e^{\ii u \Gamma(r,\theta)}]$ instead of 
 $\mathbb{E}[e^{\ii u Z}]$.
 
\subsection{\textbf{Proof of Theorem~\ref{mainresult}: The total variation distance}}\label{sub:resulttv}

In this subsection, we provide the proof of Theorem~\ref{mainresult}.

\subsubsection{\textbf{The local limit theorem for the invariant measure}}

We recall that
$X^{\e}_\infty\stackrel{d}{=}\Gamma(q_\e+1,c_\e(\infty))$, see Lemma~\ref{lem:ergodicity} in Appendix~\ref{apenddist}.
In the sequel, we show that the law of $X^\e_\infty$ satisfies a Gaussian local limit theorem.
 
\begin{lemma}[Local  limit theorem for $X^\e_\infty$]\label{lem:lcltx}
Let $m_\e:=\frac{q_\e+1}{c_\e(\infty)}=\frac{\fb}{\fa}$ and $\sigma_\e:=\frac{\sqrt{q_\e+1}}{c_\e(\infty)}=\frac{\sqrt{2\fb}}{2\fa}\e$. Then it follows that
\begin{equation}\label{eq:local}
\lim\limits_{\e\to 0^+}\tv\left(\frac{X^\e_\infty-m_\e}{\sigma_\e},\mathcal{G}\right)=0,
\end{equation}
where $\mathcal{G}$ denotes the standard Gaussian distribution.
\end{lemma}

\begin{proof}
Recall that $X^{\e}_\infty\stackrel{d}{=}\Gamma(q_\e+1,c_\e(\infty))$, where
$q_\e+1=\frac{2\fb}{\e^2}$ and $c_\e(\infty)=\frac{2\fa}{\e^2}$.
By Lemma~\ref{lem:scapropgamma} in Appendix~\ref{apen:gamma} we have 
$X^{\e}_\infty\stackrel{d}{=}c^{-1}_\e(\infty)\Gamma(q_\e+1,1)$. 
Since
\begin{equation}\label{eq:normalization}
\frac{X^\e_\infty-m_\e}{\sigma_\e}\stackrel{d}{=}\frac{\Gamma(q_\e+1,1)-(q_\e+1)}{\sqrt{q_\e+1}},
\end{equation}
Lemma~\ref{lem:lclt} in Appendix~\ref{apen:gamma} 
implies~\eqref{eq:local}. The proof is complete. 
\end{proof}

\subsubsection{\textbf{The decoupling argument in the total variation distance}}

By Lemma~\ref{lem:lcltx} it is natural to consider the normalization 
\begin{equation}\label{eq:defYt}
Y^\e_t(x):=\frac{X^\e_t(x)-m_{\e}}{\sigma_\e}\quad \textrm{ for all }\quad t\gqq 0,
\end{equation}
where $m_\e=\frac{q_\e+1}{c_\e(\infty)}=\frac{\fb}{\fa}$ and $\sigma_\e=\frac{\sqrt{q_\e+1}}{c_\e(\infty)}=\frac{\sqrt{2\fb}}{2\fa}\e$.
Straightforward computations yields
\begin{equation}\label{def:Y}
Y^\e_t(x)=\frac{c_\e(\infty)X^\e_t(x)-(q_\e+1)}{\sqrt{q_\e+1}}\quad 
\textrm{ for all }\quad t\gqq 0.
\end{equation}
Hence,  we define
\begin{equation}\label{eq:d}
\begin{split}
d^{\e,x}(t)&:=\tv\left(Y^\e_t(x),\mathcal{G}\right)\quad \textrm{ for all }\quad t\gqq 0,
\end{split}
\end{equation}
where $\mathcal{G}$ denotes the standard Gaussian distribution.

In the sequel, we show that the distance, 
\begin{equation}\label{eq:distnew}
\mathfrak{d}^{\e,x}(s_\e):=\tv\left(X^{\e}_{s_\e}(x),\mu^\e\right)
\end{equation}
 is asymptotically equivalent to $d^{\e,x}(s_\e)$
 for any function $(s_\e)_{\e>0}$ such that $s_\e\to \infty$ as $\e\to 0^+$.
Therefore, cut-off/windows cut-off/profile cut-off for
the distance $\mathfrak{d}^{\e,x}$  is equivalent for the distance $d^{\e,x}$, respectively. In other words, it is enough to show Theorem~\ref{mainresult} for the distance $d^{\e,x}$.

\begin{lemma}[Replacement lemma in the total variation distance]\label{lem:replacement} 
For any $\e\in (0,\sqrt{2\fb})$, $x\in [0,\infty)$ and $t\gqq 0$  it follows that
\begin{equation}\label{eq:decouplingnew}
|\mathfrak{d}^{\e,x}(t)-d^{\e,x}(t)|\lqq
 \tv(\sqrt{q_\e+1}\mathcal{G}+(q_\e+1),\Gamma(q_\e+1,1)),
\end{equation}
where $d^{\e,x}(t)$ and $\mathfrak{d}^{\e,x}(t)$ are 
defined in~\eqref{eq:d} 
and~\eqref{eq:distnew}, respectively.
\end{lemma}
\begin{proof}
By~Item~(i) and Item~(ii) of Lemma~\ref{lem:trasca} in Appendix~\ref{apendtv} we have
\[
d^{\e,x}(t)=\tv(c_\e(\infty)X^\e_t(x),\sqrt{q_\e+1}\mathcal{G}+(q_\e+1)).
\]
By Lemma~\ref{lem:scapropgamma} in Appendix~\ref{apen:gamma} and Item~(ii) of Lemma~\ref{lem:trasca} in Appendix~\ref{apendtv} we have
\begin{equation*}
\begin{split}
\mathfrak{d}^{\e,x}(t)&=\tv(X^\e_t(x),X^{\e}_\infty)=\tv(X^\e_t(x),\Gamma(q_\e+1,c_\e(\infty)))\\
&= \tv(X^\e_t(x),(c_\e(\infty))^{-1}\Gamma(q_\e+1,1))=\tv(c_\e(\infty)X^\e_t(x),\Gamma(q_\e+1,1)).
\end{split}
\end{equation*}
The triangle inequality for the total variation distance yields
\begin{equation}\label{ec:uno}
\begin{split}
\mathfrak{d}^{\e,x}(t)&\lqq \tv(c_\e(\infty)X^\e_t(x),\sqrt{q_\e+1}\mathcal{G}+(q_\e+1))\\
&\qquad+\tv(\sqrt{q_\e+1}\mathcal{G}+(q_\e+1),\Gamma(q_\e+1,1)).
\end{split}
\end{equation}
Similarly, 
\begin{equation}\label{ec:dos}
\begin{split}
d^{\e,x}(t)&=\tv(c_\e(\infty)X^\e_t(x),\sqrt{q_\e+1}\mathcal{G}+(q_\e+1))\\
&
\lqq \mathfrak{d}^{\e,x}(t)+ 
 \tv(\sqrt{q_\e+1}\mathcal{G}+(q_\e+1),\Gamma(q_\e+1,1)).
\end{split}
\end{equation}
By~\eqref{ec:uno} 
and~\eqref{ec:dos} we 
obtain~\eqref{eq:decouplingnew}. This completes the proof.
\end{proof}

\begin{lemma}[Asymptotically equivalent total variation distance]\label{lem:limites}
Let $x\in [0,\infty)$ be fixed. For any function $(s_\e)_{\e>0}$ such that $s_\e\to \infty$ as $\e\to 0^+$ it follows that 
\begin{equation*}
\limsup_{\e\to 0^+}\mathfrak{d}^{\e,x}(s_\e)=\limsup_{\e\to 0^+}d^{\e,x}(s_\e)
\quad \textrm{ and }\quad
\liminf_{\e\to 0^+}\mathfrak{d}^{\e,x}(s_\e)=\liminf_{\e\to 0^+}d^{\e,x}(s_\e),
\end{equation*}
where $d^{\e,x}(t)$ and $\mathfrak{d}^{\e,x}(t)$ are 
defined in~\eqref{eq:d} 
and~\eqref{eq:distnew}, respectively.
\end{lemma}
\begin{proof}
By Item~(i) and Item~(ii) of Lemma~\ref{lem:trasca} in Appendix~\ref{apendtv} we obtain
\[
\tv(\Gamma(q_\e+1,1),\sqrt{q_\e+1}\mathcal{G}+(q_\e+1))=\tv\left(\frac{\Gamma(q_\e+1,1)-(q_\e+1)}{\sqrt{q_\e+1}},\mathcal{G}\right).
\]
Then Lemma~\ref{lem:replacement} with the help of Lemma~\ref{lem:lclt} in Appendix~\ref{apen:gamma} implies that the right-hand side 
of~\eqref{eq:decouplingnew} tends to zero as $\e \to 0^+$.
The proof is complete.
\end{proof}

In the sequel, we show Theorem~\ref{mainresult} for the distance $d^{\e,x}$.
Recall that 
\begin{equation}\label{eq:dnew}
\begin{split}
d^{\e,x}(t)=\tv\left(Y^\e_t(x),\mathcal{G}\right)\quad \textrm{ for all }\quad t\gqq 0,
\end{split}
\end{equation}
where $\mathcal{G}$ denotes the standard Gaussian distribution.
We stress that $d^{\e,x}(t)$ has the advantage that its second input $\mathcal{G}$  does not depend on $\e$. 

\begin{lemma}[Limiting profile in 
 the total variation distance]\label{lem:profile}
Let $x\in [0,\infty)$ be fixed.
For $t_\e$ and $\omega_\e$ being defined in~\eqref{eq:timecut}, it follows that 
\begin{equation*}
\lim\limits_{\e\to 0^+}d^{\e,x}(t_\e+r\cdot \omega_\e)=\tv(C_x\cdot e^{- r }+\mathcal{G},\mathcal{G})\quad \textrm{ for any } \quad r \in \mathbb{R},
\end{equation*}
where $\mathcal{G}$ denotes the standard Gaussian distribution and $C_x=\frac{\sqrt{2}}{\sqrt{\fb}}(\fa x-\fb)$.
\end{lemma}
\begin{proof}
By the triangle inequality for the total variation distance, it is enough to show that
\begin{equation}\label{eq:tvYg}
\tv(Y^\e_{t_\e+r \cdot \omega_\e}(x),C_x \cdot e^{-r }+\mathcal{G})=0 \quad 
\textrm{ for any } \quad r \in \mathbb{R}.
\end{equation}
To show the preceding limit, we apply Proposition~\ref{lem:fouriertv}  in Appendix~\ref{apendtv}. 
Then we need to show Item~(i) Convergence in distribution, Item~(ii) Fourier integrability and Item~(iii) Tail behavior.

Now, we compute the  characteristic function of the marginal $Y_t^{\e}(x)$. By Corollary~\ref{characteristic function} in Appendix~\ref{apenddist} we have 
\begin{equation}\label{eq:Yfourier}
\mathbb E[e^ {\ii z Y_{t}^{\e }(x)}]= \psi\left(\frac{c_\e(\infty)z}{\sqrt{q_{\e}+1}};t,\e,x \right)e^{-\ii z \sqrt{q_{\e}+1}},\quad z\in \mathbb{R},
\end{equation}
where
\begin{equation}\label{eq:exponente}
\begin{split}
\psi(z;t,\e, x):&=\mathbb E[e^{\ii z X_t^{\e}(x)}]\\
&=\frac{1}{\left(1-\ii\frac{z}{c_\e(t)}\right)^{q_\e+1}} \exp \left(\ii\frac{z c^2_\e(t)xe^{-\fa t}}{c^2_\e(t)+z^2}\right)\exp \left(-\frac{z^2 c_\e(t)xe^{-\fa t}}{c^2_\e(t)+z^2}\right)
\end{split}
\end{equation}
for any $z\in \mathbb{R}$.
Let 
$\varphi_\e(z;x):=\mathbb E[e^ {\ii z Y_{t_\e+r\cdot \omega_\e}^{\e }(x)}]$ for any $z\in \mathbb{R}$, 
where 
$t_\e$ and $\omega_\e$ are defined in~\eqref{eq:timecut}, and $r\in \mathbb{R}$.
Note that Lemma~\ref{Lem:limit 1} in Appendix~\ref{apenddist} implies 
$$\lim_{\e\to 0}\varphi_\e(z;x)=\exp\left(-\frac{z^2}{2}\right)\exp\left(  z \frac{\sqrt{2}}{\sqrt{\fb}}(\fa x-\fb)e^{-r}\right),$$ which yields Item~(i).

We continue with the proof of Item~(ii).
We note that 
\[
\frac{c^2_\e(\infty)}{c^2_\e(t)}\frac{1}{q_\e+1}=\frac{(1-e^{-\fa t})^2\e^2}{2\fb}\quad \textrm{ for any }\quad t>0.
\]
By~\eqref{eq:Yfourier} 
and~\eqref{eq:exponente} for any $r \in \mathbb{R}$ we have
\begin{equation}\label{eq:cotasuperiorF}
\int_{\mathbb{R}} |\varphi_\e(z;x)|\ud z\lqq 
\int_{\mathbb{R}} 
\frac{1}{\left(1+\frac{(1-\e e^{-\fa r })^2\e^2}{2\fb} z^2\right)^{\frac{\fb}{\e^2}}}
\ud z
<\infty\quad \textrm{ when }\quad  \e\in (0,\sqrt{2\fb}),
\end{equation}
where 
$t_\e$ and $\omega_\e$ are defined in~\eqref{eq:timecut}.
It is easy to see that $\int_{\mathbb{R}}|\varphi(z;x)|\ud z< \infty$.
The proof of Item~(ii) is complete. 

Finally, we verify Item~(iii).
Observe that 
\[
\frac{1}{8}\lqq \frac{(1-\e e^{-r})^2}{2}\quad \textrm{ for }\quad  \e\in \left(0,\frac{1}{2}e^{r}\right).
\]
By~\eqref{eq:cotasuperiorF} we obtain
\begin{equation}\label{eq:phipa}
\begin{split}
\int_{|z|\gqq \ell} |\varphi_\e(z;x)|\ud z&\lqq \int_{|z|\gqq \ell} 
\frac{1}{\left(1+\frac{z^2}{8}\frac{\e^2}{\fb}\right)^{\frac{\fb}{\e^2}}}
\ud z\lqq \int_{|z|\gqq \ell} 
\frac{1}{\left(1+\frac{(z^2/8)}{\frac{\fb}{\e^2_0}}\right)^{\frac{\fb}{\e^2_0}}}
\ud z<\infty
\end{split}
\end{equation}
for any $\e\in (0, \e_0)$, where 
$\e_0:=\frac{1}{2}\min\{\frac{1}{2}e^{r},\sqrt{2\fb}\}$.
Then the Dominated Convergence Theorem yields 
\begin{equation*}
\begin{split}
\lim\limits_{\e\to 0^+}\int_{|z|\gqq \ell} |\varphi_\e(z;x)|\ud z
=\int_{|z|\gqq \ell} 
e^{-\frac{z^2}{8}}
\ud z,
\end{split}
\end{equation*}
which implies
\[
\lim\limits_{\ell\to \infty}\limsup\limits_{\e\to 0^+}\int_{|z|\gqq \ell} |\varphi_\e(z;x)|\ud z=0.
\]
This complete the proof of Item~(iii)
and therefore the proof of Lemma~\ref{lem:profile} is complete.
\end{proof}

Now, we analyze the case $x=\fb/\fa$ that yields $C_x=0$. Then Lemma~\ref{lem:profile} does not imply profile cut-off phenomenon at $t_\e+r\cdot \omega_\e$. In fact, in this case, there is no cut-off phenomenon as the following lemma states.
\begin{lemma}[No cut-off phenomenon in the total variation distance for $x=\fb/\fa$]\label{lem:nocutoff}
Let $x=\fb/\fa$ be fixed.
For any function $(s_\e)_{\e>0}$ satisfying $s_\e\to \infty$ as $\e\to 0^+$, it follows that 
\begin{equation*}
\lim\limits_{\e\to 0^+}\tv(X^\e_{s_\e}(x),X^\e_\infty)=0.
\end{equation*}
\end{lemma}
\begin{proof}
By Lemma~\ref{lem:replacement} it is enough to show
$\lim_{\e\to 0^+}d^{\e,x}(s_\e)=0$,
where $d^{\e,x}$ is defined 
in~\eqref{eq:dnew}.
Similarly to the proof of Lemma~\ref{lem:profile},
to show the preceding limit, we apply Proposition~\ref{lem:fouriertv}  in Appendix~\ref{apendtv}. 
Then we need to show Item~(i) Convergence in distribution, Item~(ii) Fourier integrability and Item~(iii) Tail behavior.

Let $\widetilde{\varphi}_\e(z;x):=\mathbb E[e^ {\ii z Y_{s_\e}^{\e }(x)}]$ for any  $z\in \mathbb{R}$. Note that Lemma~\ref{Lemma lim 2} in Appendix~\ref{apenddist}
implies

\begin{equation}\label{eq:tildelimite}
\lim\limits_{\e\to 0^+}\widetilde{\varphi}_\e(z;x)=e^{-z^2/2}\quad \textrm{ for all }\quad z\in \mathbb{R}.
\end{equation}
This completes the proof of Item~(i).

We continue with the proof of Item~(ii).
Analogously 
to~\eqref{eq:cotasuperiorF} we obtain
\begin{equation*}
\int_{\mathbb{R}} |\widetilde{\varphi}_\e(z;x)|\ud z\lqq 
\int_{\mathbb{R}} 
\frac{1}{\left(1+\frac{(1-e^{-\fa s_\e})^2\e^2}{2\fb} z^2\right)^{\frac{\fb}{\e^2}}}
\ud z
<\infty\quad \textrm{ when }\quad  \e\in (0,\sqrt{2\fb}),
\end{equation*}
implying Item~(ii).

Finally, Item~(iii) follows similarly to~\eqref{eq:phipa}. Indeed, since $s_\e \to \infty$ as $\e\to 0^+$, we have the existence of $\widetilde{\e}_0:=\widetilde{\e}_0(\fa)$ such that
\[
\frac{1}{8}\lqq \frac{(1-e^{-\fa s_\e})^2}{2}\quad \textrm{ for }\quad  \e\in (0,\min\{\widetilde{\e}_0,\sqrt{2\fb}\}).
\]
Then we have
\begin{equation}\label{eq:phipanew}
\begin{split}
\int_{|z|\gqq \ell} |\widetilde{\varphi}_\e(z;x)|\ud z&\lqq \int_{|z|\gqq \ell} 
\frac{1}{\left(1+\frac{z^2}{8}\frac{\e^2}{\fb}\right)^{\frac{\fb}{\e^2}}}
\ud z\lqq \int_{|z|\gqq \ell} 
\frac{1}{\left(1+\frac{(z^2/8)}{\frac{\fb}{\e^2_0}}\right)^{\frac{\fb}{\e^2_0}}}
\ud z<\infty,
\end{split}
\end{equation}
which implies Item~(iii)
and therefore the proof of Lemma~\ref{lem:nocutoff} is complete.
\end{proof}

\subsubsection{\textbf{Proof of Theorem~\ref{mainresult}}}\label{sub:WAproof}

In this subsection, we stress the fact that Theorem~\ref{mainresult} is just a consequence of what we have already proved up to here.

Let $x\in [0,\infty)$ be fixed.
For $t_\e$ and $\omega_\e$ being defined in~\eqref{eq:timecut},
Lemma~\ref{lem:profile} with the help of Lemma~\ref{lem:replacement} and Lemma~\ref{lem:limites} implies
\begin{equation*}
\lim\limits_{\e\to 0^+}\mathfrak{d}^{\e,x}(t_\e+r\cdot \omega_\e)=\tv(C_x \cdot e^{- r }+\mathcal{G},\mathcal{G})\quad \textrm{ for any } \quad r \in \mathbb{R},
\end{equation*}  
where
$\mathcal{G}$ denotes the standard Gaussian distribution and $C_x=\frac{\sqrt{2}}{\sqrt{\fb}}(\fa x-\fb)$.
This yields~\eqref{eq:Gx}.
For $x\gqq 0$ and $x\neq \fb/\fa$ we have that $C_x\neq 0$, which implies~\eqref{eq:inftylimites}.
and therefore~\eqref{def:cut3} holds true.
Finally, 
Lemma~\ref{lem:nocutoff} yields no cut-off phenomenon at $x=\fb/\fa$ in the sense of~\eqref{def:cut1}.
The proof is complete.

\subsection{\textbf{Proof of Theorem~\ref{mainresultW}: The Wasserstein distance}}\label{subsec TW}
In this section, we show Theorem~\ref{mainresultW}.

We start recalling two basic properties for the Wasserstein distance that are useful during this section.
\begin{lemma}[Basic properties of the Wasserstein distance of order $p>0$]\label{lem:propWA}  Let $X$ and $Y$ be random vectors taking values in $\mathbb{R}^d$. 
Assume that $X$ and $Y$ have finite $p$-th moment for some $p>0$.
Then the following is valid.
\begin{itemize}
\item[(i)] Translation invariance: For any deterministic vectors $v_1,v_2\in \mathbb{R}^d$ it follows that
\[
\mathcal{W}_p(v_1+X, v_2+Y)=\mathcal{W}_p(v_1-v_2+X,Y)=\mathcal{W}_p(X,v_2-v_1+Y).
\]
\item[(ii)] Homogeneity: For any non-zero constant $c$ it follows that
\[
\mathcal{W}_p(cX, cY)=|c|^{1\wedge p}\mathcal{W}_p(X,Y)=
\begin{cases}
|c|\mathcal{W}_p(X,Y) & \textrm{ for }\quad p\in [1,\infty),\\
|c|^p\mathcal{W}_p(X,Y) & \textrm{ for }\quad p\in (0,1].
\end{cases}
\]
\end{itemize}
\end{lemma}

\begin{proof}
For details, we refer  to Lemma~2.2 
in~\cite{BHPWA}.
\end{proof}

\subsubsection{\textbf{The limit theorem for the invariant measure}}
In this subsection, similarly as Lemma~\ref{lem:lcltx}, we prove that the marginal $X^\e_\infty$ satisfies
a Gaussian limit theorem in the Wasserstein distance. 

\begin{lemma}[Limit theorem for $X^\e_\infty$ in $\mathcal{W}_p$ for any $p>0$]\label{lem:lcltxW}
Let $m_\e:=\frac{q_\e+1}{c_\e(\infty)}=\frac{\fb}{\fa}$ and $\sigma_\e:=\frac{\sqrt{q_\e+1}}{c_\e(\infty)}=\frac{\sqrt{2\fb}}{2\fa}\e$. Then for any $p>0$ it follows that
\begin{equation}\label{eq:localw}
\lim\limits_{\e\to 0^+}\mathcal{W}_p\left(\frac{X^\e_\infty-m_\e}{\sigma_\e},\mathcal{G}\right)=0,
\end{equation}
where $\mathcal{G}$ denotes the standard Gaussian distribution.
\end{lemma}

\begin{proof}
By~\eqref{eq:normalization} we have that
\begin{equation}\label{eq:gammae}
Y^\e_\infty:=\frac{X^\e_\infty-m_\e}{\sigma_\e}\stackrel{d}{=}\frac{\Gamma(q_\e+1,1)-(q_\e+1)}{\sqrt{q_\e+1}}
\end{equation}
Lemma~\ref{lem:lcltx} yields that $\frac{X^\e_\infty-m_\e}{\sigma_\e}$ converges in distribution to $\mathcal{G}$ as $\e\to 0^+$.
Hence, by Theorem~7.12 
in~\cite{Villani2009} it is enough to show that the family $(Y^\e_\infty)_{\e\in (0,\sqrt{2\fb})}$ is uniformly integrable, that is,
\begin{equation}\label{eq:uiw}
\lim_{N\to \infty} \limsup_{\e\to 0^+}\mathbb{E}[|Y^\e_\infty|^p \mathbf{1}_{\{|Y^\e_\infty|\gqq N\}}]=0.
\end{equation}
We start with the following observation.
For any $|\lambda|>0$ and $p>0$ there exists a positive constant $C_{\lambda,p}$
such that 
\begin{equation}\label{eq:expine}
|z|^p\lqq C_{\lambda,p} e^{|\lambda| |z|}\lqq C_{\lambda,p} e^{\lambda z}+C_{\lambda,p} e^{-\lambda z}\quad \textrm{ for any } \quad z\in \mathbb{R}.
\end{equation}
Then the Cauchy--Schwarz inequality with the help of subadditivity of the square root map $[0,\infty)\ni r\mapsto \sqrt{r}$ yields
\begin{equation}\label{eq:cssub}
\begin{split}
\mathbb{E}[|Y^{\e}_\infty|^p \mathbf{1}_{\{|Y^\e_\infty|\gqq N\}}]&\lqq \sqrt{\mathbb{E}[|Y^{\e}_\infty|^{2p}]}
\sqrt{\mathbb{P}\left(|Y^\e_\infty|\gqq N\right)}\\
&\lqq \sqrt{C_{\lambda,2p}\mathbb{E}[ e^{\lambda Y^{\e}_\infty }]+ C_{\lambda,2p}\mathbb{E}[e^{-\lambda Y^{\e}_\infty}]}\sqrt{\mathbb{P}\left(|Y^\e_\infty|\gqq N\right)}\\
&\lqq \sqrt{C_{\lambda,2p}}\left(\sqrt{\mathbb{E}[e^{\lambda Y^{\e}_\infty}]}+\sqrt{\mathbb{E}[e^{-\lambda Y^{\e}_\infty}]}\right)\sqrt{\mathbb{P}\left(|Y^\e_\infty|\gqq N\right)}.
\end{split}
\end{equation}
We claim that 
\begin{equation}\label{eq:claimui}
\lim_{\e\to 0^+}\mathbb{E}[e^{\lambda Y^{\e}_\infty}]=e^{\lambda^2/2}=:K_\lambda\quad \textrm{ for any }\quad \lambda\in \mathbb{R}. 
\end{equation}
Then for any $|\lambda|>0$ we have
\begin{equation}\label{eq:limsupnueva1}
\begin{split}
\limsup_{\e\to 0^+}\mathbb{E}[|Y^{\e}_\infty|^p \mathbf{1}_{\{|Y^\e_\infty|\gqq N\}}]&\lqq \sqrt{C_{\lambda,2p}}\limsup_{\e\to 0^+}(\sqrt{\mathbb{E}[e^{\lambda Y^{\e}_\infty}]}+\sqrt{\mathbb{E}[e^{-\lambda Y^{\e}_\infty}]})\sqrt{\mathbb{P}\left(|Y^\e_\infty|\gqq N\right)}\\
&\lqq 2\sqrt{C_{\lambda,2p}K_{\lambda}}
\limsup_{\e\to 0^+}\sqrt{\mathbb{P}\left(|Y^\e_\infty|\gqq N\right)}\\
&= 2\sqrt{C_{\lambda,2p}K_{\lambda}}
\sqrt{\mathbb{P}\left(|\mathcal{G}|\gqq N\right)},
\end{split}
\end{equation}
where in the last equality we have used that $Y^\e_\infty$ converges in distribution to $\mathcal{G}$ as $\e\to 0^+$.
Now, sending $N\to \infty$ we obtain~\eqref{eq:uiw}.
In the sequel, we 
show~\eqref{eq:claimui}.
By~\eqref{eq:gammae} we have
\begin{equation}\label{eq:limsup1}
\begin{split}
\mathbb{E}[e^{\lambda Y^\e_\infty}]&= 
\mathbb{E}[e^{\lambda \frac{\Gamma(q_\e+1,1)}{ \sqrt{q_\e+1}}}]
e^{-\lambda \sqrt{q_\e+1}}=\mathbb{E}[e^{\frac{\lambda}{\sqrt{q_\e+1}}\Gamma(q_\e+1,1)}]
e^{-\lambda \sqrt{q_\e+1}}\\
&=\left(1-\frac{\lambda}{\sqrt{q_\e+1}}\right)^{-(q_\e+1)}e^{-\lambda \sqrt{q_\e+1}}=\frac{1}{\left(1+\frac{-\lambda}{\sqrt{q_\e+1}}\right)^{(q_\e+1)}}\frac{1}{e^{-\sqrt{q_\e+1}(-\lambda)}}
\end{split}
\end{equation}
for any $\lambda<\sqrt{q_\e+1}$.
By Lemma~\ref{lem:logconv} in Appendix~\ref{ap:tools} we obtain~\eqref{eq:claimui}.
The proof is complete. 
\end{proof}

\subsubsection{\textbf{The decoupling argument}}

In the sequel, we estimate
\[
\frac{1}{\e^{1\wedge p}}\mathcal{W}_p(X^\e_t(x),X^\e_\infty).
\]
By Item~(ii) in Lemma~\ref{lem:propWA} and the triangle inequality for $\mathcal{W}_p$ we have
\begin{equation}
\begin{split}
\frac{1}{\e^{1\wedge p}}\mathcal{W}_p(X^\e_t(x),X^\e_\infty)&=
\frac{\sigma^{1\wedge p}_\e}{\e^{1\wedge p}}\mathcal{W}_p\left(\frac{X^\e_t(x)-m_\e}{\sigma_\e},\frac{X^\e_\infty-m_\e}{\sigma_\e}\right)\\
&\hspace{-3cm}\lqq \left(
\frac{\sqrt{2\fb}}{2\fa}\right)^{1\wedge p}\mathcal{W}_p\left(\frac{X^\e_t(x)-m_\e}{\sigma_\e},\mathcal{G}\right)
+\left(\frac{\sqrt{2\fb}}{2\fa}\right)^{1\wedge p}\mathcal{W}_p\left(\mathcal{G},\frac{X^\e_\infty-m_\e}{\sigma_\e}\right),
\end{split}
\end{equation}
where in the last inequality we have used that $\sigma_\e=\frac{\sqrt{2\fb}}{2\fa}\e$
and $\mathcal{G}$ denotes the standard Gaussian distribution.
Similarly, we obtain
\begin{equation}
\begin{split}
&\left|\frac{1}{\e^{1\wedge p}}\mathcal{W}_p(X^\e_t(x),X^\e_\infty)-
\left(\frac{\sqrt{2\fb}}{2\fa}\right)^{1\wedge p}\mathcal{W}_p\left(\frac{X^\e_t(x)-m_\e}{\sigma_\e},\mathcal{G}\right)\right|\\
&\hspace{8cm}\lqq
\left(\frac{\sqrt{2\fb}}{2\fa}\right)^{1\wedge p}\mathcal{W}_p\left(\frac{X^\e_\infty-m_\e}{\sigma_\e},\mathcal{G}\right).
\end{split}
\end{equation}

The preceding inequality yields the following replacement lemma.
\begin{lemma}[Replacement lemma in the Wasserstein distance of order $p>0$]
\label{lem:reemplazoKRW}
Let $p>0$ be fixed. For any $x\gqq 0$ and $t>0$ it follows that 
\begin{equation}
\begin{split}
&\left|\frac{1}{\e^{1\wedge p}}\mathcal{W}_p(X^\e_t(x),X^\e_\infty)-
\left(\frac{\sqrt{2\fb}}{2\fa}\right)^{1\wedge p}\mathcal{W}_p\left(\frac{X^\e_t(x)-m_\e}{\sigma_\e},\mathcal{G}\right)\right|\\
&\hspace{8cm}\lqq
\left(\frac{\sqrt{2\fb}}{2\fa}\right)^{1\wedge p}\mathcal{W}_p\left(\frac{X^\e_\infty-m_\e}{\sigma_\e},\mathcal{G}\right),
\end{split}
\end{equation}
where $\mathcal{G}$ denotes the standard Gaussian distribution.
\end{lemma}

By Lemma~\ref{lem:lcltxW} and Lemma~\ref{lem:reemplazoKRW} we obtain the following lemma.
\begin{lemma}[Asymptotically equivalent distances in the Wasserstein distance of order $p>0$]\label{lem:limitesKWR}
Let  $p>0$ and $x\gqq 0$ be fixed. For any function $(s_\e)_{\e>0}$ such that $s_\e\to \infty$ as $\e\to 0^+$ it follows that 
\begin{equation}
\limsup_{\e\to 0^+}\frac{1}{\e^{1\wedge p}}\mathcal{W}_p(X^\e_{s_\e}(x),X^\e_\infty)=\left(\frac{\sqrt{2\fb}}{2\fa}\right)^{1\wedge p}\limsup_{\e\to 0^+}\mathcal{W}_p\left(\frac{X^\e_{s_\e}(x)-m_\e}{\sigma_\e},\mathcal{G}\right)
\end{equation}
and 
\begin{equation}
\liminf_{\e\to 0^+}\frac{1}{\e^{1\wedge p}}\mathcal{W}_p(X^\e_{s_\e}(x),X^\e_\infty)=\left(\frac{\sqrt{2\fb}}{2\fa}\right)^{1\wedge p}\liminf_{\e\to 0^+}\mathcal{W}_p\left(\frac{X^\e_{s_\e}(x)-m_\e}{\sigma_\e},\mathcal{G}\right).
\end{equation}
\end{lemma}

Finally, the limiting profile is calculated in the following lemma.
\begin{lemma}[Limiting profile]\label{lem:profileWA}
Let $p>0$ and $x\gqq 0$ be fixed.
For $t_\e$ and $\omega_\e$ being defined in~\eqref{eq:timecut}, it follows that 
\begin{equation*}
\lim\limits_{\e\to 0^+} \frac{1}{\e^{1\wedge p}}\mathcal{W}_p(X^\e_{t_\e+r\cdot \omega_\e}(x),X^\e_\infty)
=\left(\frac{\sqrt{2\fb}}{2\fa}\right)^{1\wedge p}\mathcal{W}_p(C_x\cdot e^{- r }+\mathcal{G},\mathcal{G})
\quad \textrm{ for any } \quad r \in \mathbb{R},
\end{equation*}
where $\mathcal{G}$ denotes the standard Gaussian distribution and $C_x=\frac{\sqrt{2}}{\sqrt{\fb}}(\fa x-\fb)$.
\end{lemma}

\begin{proof}
Recall by~\eqref{eq:defYt} 
and~\eqref{def:Y} that
\[
Y^\e_{t}(x)=\frac{X^\e_{t}(x)-m_\e}{\sigma_\e}=\frac{c_\e(\infty)X^\e_t(x)-(q_\e+1)}{\sqrt{q_\e+1}}\quad \textrm{ for all }\quad t\gqq 0.
\]
By Lemma~\ref{lem:limitesKWR} it is enough to show that 
\begin{equation}\label{eq:claminnew}
\lim\limits_{\e\to 0^+}\mathcal{W}_p\left(Y^\e_{t_\e+r\cdot \omega_\e}(x),\mathcal{G}\right)= \mathcal{W}_p\left(C_x \cdot e^{-r}+\mathcal{G},\mathcal{G}\right)\quad \textrm{ for any }\quad r\in \mathbb{R}.
\end{equation}
By the triangle inequality for $\mathcal{W}_p$ we note that 
\[
\mathcal{W}_p\left(Y^\e_{t_\e+r\cdot \omega_\e}(x),\mathcal{G}\right)\lqq 
\mathcal{W}_p\left(Y^\e_{t_\e+r\cdot \omega_\e}(x),C_x \cdot e^{-r}+\mathcal{G}\right)+\mathcal{W}_p\left(C_x \cdot e^{-r}+\mathcal{G},\mathcal{G}\right)
\]
and 
\[
\mathcal{W}_p\left(\mathcal{G},C_x \cdot e^{-r}+\mathcal{G}\right)
\lqq 
\mathcal{W}_p\left(\mathcal{G},Y^\e_{t_\e+r\cdot \omega_\e}(x)\right)+
\mathcal{W}_p\left(Y^\e_{t_\e+r\cdot \omega_\e}(x),C_x \cdot e^{-r}+\mathcal{G}\right).
\]
Therefore, to 
prove~\eqref{eq:claminnew} it is enough to show that 
\begin{equation}\label{eq:rane}
\lim_{\e\to 0^+}\mathcal{W}_p\left(Y^\e_{t_\e+r\cdot \omega_\e}(x),C_x\cdot e^{-r}+\mathcal{G}\right)=0 \quad \textrm{ for any }\quad r\in \mathbb{R}.
\end{equation}
The proof of~\eqref{eq:rane} is analogous to the proof of Lemma~\ref{lem:lcltxW}.
Indeed, since
$Y^\e_{t_\e+r\cdot \omega_\e}(x)$ converges in distribution to $C_x\cdot e^{- r }+\mathcal{G}$ as $\e\to 0^+$, see~\eqref{eq:tvYg}, 
Theorem~7.12 in~\cite{Villani2009} implies that  
it is enough to show that the family
$(Y^\e_{t_\e+r\cdot \omega_\e}(x))_{\e\in (0,\sqrt{2\fb}}$ is uniformly integrable.
More precisely, 
\begin{equation}\label{eq:uiw7}
\lim_{N\to \infty} \limsup_{\e\to 0^+}\mathbb{E}[|Y^\e_{t_\e+r\cdot \omega_\e}(x)|^p \mathbf{1}_{\{|Y^\e_{t_\e+r\cdot \omega_\e}(x)|\gqq N\}}]=0.
\end{equation}
Similarly to~\eqref{eq:cssub} we obtain
\begin{equation}\label{eq:UIYt}
\begin{split}
\mathbb{E}\left[|Y^\e_{t_\e+r\cdot \omega_\e}(x)|^p \mathbf{1}_{\{|Y^\e_{t_\e+r\cdot \omega_\e}(x)|\gqq N\}}\right]
&\lqq \sqrt{C_{\lambda,2p}} \left(\sqrt{\mathbb{E}[e^{\lambda Y^\e_{t_\e+r\cdot \omega_\e}(x)}]}+\sqrt{\mathbb{E}[e^{-\lambda Y^\e_{t_\e+r\cdot \omega_\e}(x)}]}\right)\\
&\qquad \times \sqrt{\mathbb{P}\left(|Y^\e_{t_\e+r\cdot \omega_\e}(x)|\gqq N\right)},
\end{split}
\end{equation}
where the constant $C_{\lambda,2p}$ is given in~\eqref{eq:expine}.
With the help of Lemma \ref{Lem:limit 1} in Appendix~\ref{apenddist} we obtain
\begin{equation*}
\lim_{\e\to 0^+}\mathbb{E}[e^{\lambda Y^\e_{s_\e}(x)}]=e^{\lambda^2/2}e^{\lambda C_x}\quad \textrm{ for any }\quad \lambda\in \mathbb{R}. 
\end{equation*}
An analogous reasoning using 
in~\eqref{eq:limsupnueva1} 
yields~\eqref{eq:uiw7}.
The proof is complete.
\end{proof}

In the sequel, we analyze the case $x=\fb/\fa$ that yields $C_x=0$. Hence, Lemma~\ref{lem:profileWA} does not imply profile cut-off phenomenon at $t_\e+r\cdot \omega_\e$. In fact, in this case, there is no cut-off phenomenon as the following lemma states.
\begin{lemma}[No cut-off phenomenon for $x=\fb/\fa$ in $\mathcal{W}_p$ for $p>0$]\label{lem:nocutoffWA}
Let  $p>0$ be fixed. For $x=\fb/\fa$ and
for any function $(s_\e)_{\e>0}$ satisfying $s_\e\to \infty$ as $\e\to 0^+$, it follows that 
\begin{equation}\label{eq:demosno}
\lim\limits_{\e\to 0^+}\frac{1}{\e^{1\wedge p}}\mathcal{W}_p(X^\e_{s_\e}(x),X^\e_\infty)=0.
\end{equation}
\end{lemma}

\begin{proof}
By Lemma~\ref{lem:limitesKWR}, the limit~\eqref{eq:demosno} is equivalent to the limit
\begin{equation}\label{eq:litint}
\lim_{\e\to 0^+}\mathcal{W}_p\left(Y^\e_{s_\e}(x),\mathcal{G}\right)=0,
\end{equation}
where $\mathcal{G}$ denotes the Gaussian distribution and
\[
Y^\e_{t}(x)=\frac{X^\e_{t}(x)-m_\e}{\sigma_\e}=\frac{c_\e(\infty)X^\e_t(x)-(q_\e+1)}{\sqrt{q_\e+1}}\quad \textrm{ for all }\quad t\gqq 0.
\]
By~\eqref{eq:tildelimite} we have that $Y^\e_{s_\e}(x)$ tends in distribution to $\mathcal{G}$ as $\e\to 0^+$. To 
conclude~\eqref{eq:litint} it is enough to show that the family $(Y^\e_{s_\e}(x))_{\e\in (0,\sqrt{2\fb})}$ is uniformly integrable.

Similarly to~\eqref{eq:cssub} we obtain
\begin{equation}\label{eq:UIYtnew}
\begin{split}
\mathbb{E}\left[|Y^\e_{s_\e}(x)|^p \mathbf{1}_{\{|Y^\e_{s_\e}(x)|\gqq N\}}\right]
&\lqq \sqrt{C_{\lambda,2p}} \left(\sqrt{\mathbb{E}[e^{\lambda Y^\e_{s_\e}(x)}]}+\sqrt{\mathbb{E}[e^{-\lambda Y^\e_{s_\e}(x)}]}\right)\\
&\qquad \times \sqrt{\mathbb{P}\left(|Y^\e_{s_\e}(x)|\gqq N\right)}
\end{split}
\end{equation}
for $|\lambda|>0$,
where the constant $C_{\lambda,2p}$ is given in~\eqref{eq:expine}.
With the help of Lemma~\ref{Lemma lim 2}  in Appendix~\ref{apenddist} we obtain
\begin{equation*}
\lim_{\e\to 0^+}\mathbb{E}[e^{\lambda Y^\e_{s_\e}(x)}]=e^{\lambda^2/2}\quad \textrm{ for any }\quad \lambda\in \mathbb{R}. 
\end{equation*}
An analogous reasoning using 
in~\eqref{eq:limsupnueva1} 
yields that the family $(Y^\e_{s_\e}(x))_{\e\in (0,\sqrt{2\fb})}$ is uniformly integrable.
The proof is complete.
\end{proof}

\subsubsection{\textbf{Proof of Theorem~\ref{mainresultW}}}

In this subsection, we stress the fact that Theorem~\ref{mainresultW} is just a consequence of what we have already proved up to here.

Let $p>0$ be fixed.
For $x\gqq 0$ and for  $t_\e$ and $\omega_\e$ being defined 
in~\eqref{eq:timecut},
Lemma~\ref{lem:profileWA} implies
\begin{equation*}
\lim\limits_{\e\to 0^+} \frac{1}{\e^{1\wedge p}}\mathcal{W}_p(X^\e_{t_\e+r\cdot \omega_\e}(x),X^\e_\infty)
=\left(\frac{\sqrt{2\fb}}{2\fa}\right)^{1\wedge p}\mathcal{W}_p(C_x\cdot e^{- r }+\mathcal{G},\mathcal{G})
\quad \textrm{ for any } \quad r \in \mathbb{R},
\end{equation*}
where $\mathcal{G}$ denotes the standard Gaussian distribution and $C_x=\frac{\sqrt{2}}{\sqrt{\fb}}(\fa x-\fb)$.

This yields~\eqref{eq:GxW}.
For $x\gqq 0$ and $x\neq \fb/\fa$ we have that $C_x\neq 0$, which implies~\eqref{eq:inftylimitesW}.
and therefore~\eqref{def:cut3} holds true.

Finally, 
Lemma~\ref{lem:nocutoffWA} yields no cut-off phenomenon at $x=\fb/\fa$ in the sense 
of~\eqref{def:cut1}. The proof is complete.

\appendix

\section{\textbf{General properties of the total variation distance and 
Parseval--Plancherel--Fourier approach to total variation convergence}}\label{apendtv}

In this section, we provide general properties of the total variation distance. 
As we have already discussed in Section~\ref{sec:introd}, in general, convergence in distribution does not imply convergence in the total variation distance.
We then impose suitable hypotheses on the sequence of characteristic functions that ensure convergence in the total variation distance. 
Since the latter could be of independent interest, we state the results in a multidimensional setting.

Along this section, $d$ always is a positive fixed integer, $\|\cdot\|$ denotes the Euclidean norm in $\mathbb{R}^d$, $\langle \cdot,\cdot \rangle$ denotes the standard inner product in $\mathbb{R}^d$ and $\mathcal{B}(\mathbb{R}^d)$ are the Borelian sets of $\mathbb{R}^d$.

\begin{lemma}[General properties of the total variation distance]\label{lem:trasca}
Let $X$ and $Y$ be random vectors defined in the probability space  $(\Omega,\mathcal{F},\mathbb{P})$ and taking values in $\mathbb{R}^d$. Then the following is valid.
\begin{itemize}
\item[(i)] Translation invariance: For any deterministic vectors $v_1,v_2\in \mathbb{R}^d$ it follows that
\[
\tv(v_1+X,v_2+Y)=\tv(v_1-v_2+X,Y)=\tv(X,v_2-v_1+Y).
\]
\item[(ii)] Scaling invariance: For any  non-zero constant $c$ it follows that
\[
\tv(cX,cY)=\tv(X,Y). 
\]
\end{itemize}
\end{lemma}
The proof is a direct consequence of Theorem~5.2 in~\cite{Devroyeetalt2001}.

\begin{lemma}[Scheff\'e's Lemma]\label{lem:scheffe}
Let $(\Omega,\mathcal{F},\mathbb{P})$ be a probability space
and consider 
a sequence $(\mathbb{P}_n)_{n\in \mathbb{N}}$ of probability measures on space $(\Omega,\mathcal{F})$. 
For each $n\in \mathbb{N}$
assume that $f_n:\mathbb{R}^d\to [0,\infty)$ is a density of $\mathbb{P}_n$ with respect to the Lebesgue measure on 
$(\mathbb{R}^d, \mathcal{B}(\mathbb{R}^d))$.
In addition, assume that 
$f:\mathbb{R}^d\to [0,\infty)$ is a density of $\mathbb{P}$ with respect to the Lebesgue measure on 
$(\mathbb{R}^d, \mathcal{B}(\mathbb{R}^d))$.
If $\lim\limits_{n\to \infty}f_n(x)=f(x)$ for $x$-almost everywhere (w.r.t. the Lebesgue measure), then it follows that
\[
\lim\limits_{n\to \infty}\tv(\mathbb{P}_n,\mathbb{P})=0. 
\]
\end{lemma}
The proof is well-known and it can be found in Lemma~3.3.2 of~\cite{Reiss1989}.
\vspace{0.3cm}

Harmonic analysis and representation theory have been applied to bound the total variation distance to equilibrium in compact groups, and in particular, in finite groups, see Chapter~2 in~\cite{Ceccheetalt} and Chapter~3 in~\cite{PDIbook}.
In Proposition~\ref{lem:fouriertv} and Proposition~\ref{prof:parseval} below, 
using Fourier inversion and the Parseval--Plancherel isometry, we 
obtain total variation convergence in $\mathbb{R}^d$.
\begin{proposition}[Fourier approach to total variation convergence]\label{lem:fouriertv}
Let $(X_\e)_{\e>0}$ be a family of random vectors defined in the probability space $(\Omega,\mathcal{F},\mathbb{P})$ and taking values on $(\mathbb{R}^d,\mathcal{B}(\mathbb{R}^d))$.
Let $X$ be a random vector defined in $(\Omega,\mathcal{F},\mathbb{P})$ and taking values on $(\mathbb{R}^d,\mathcal{B}(\mathbb{R}^d))$.
Assume that 
\begin{itemize}
\item[(i)] Convergence in distribution:
$X_\e\stackrel{d}{\longrightarrow} X$ as $\e\to 0^+$.
\item[(ii)] Fourier integrability:
\begin{equation}\label{eq:fourier}
\varphi_\e,\varphi \in L^1(\mathbb{R}^d)\quad \textrm{ for each }\quad \e>0,
\end{equation}
where $\varphi_\e$ and $\varphi$ denote the characteristic function of $X_\e$ and $X$, respectively.
\item[(iii)] Tail behavior (uniform integrability):
\begin{equation}\label{tail}
\limsup_{\ell \to \infty}\limsup_{\e\to 0^+}\int_{\{|z|\gqq \ell\}}|\varphi_\e(z)|\ud z=0.
\end{equation}
\end{itemize}
Then it follows that
\begin{equation}\label{eq:tvlimit}
\lim\limits_{\e\to 0^+} \tv(X_\e,X)=0.
\end{equation}
\end{proposition}

\begin{remark}[$L^{1}(\mathbb{R}^d)$-convergence of the Fourier transform]\label{rem:L1}
Assume the hypothesis of Proposition~\ref{lem:fouriertv}.
Following step by step the proof of
Proposition~\ref{lem:fouriertv}
one can see that
\[
\lim\limits_{\e\to 0^+}\int_{\mathbb{R}^d} |\varphi_\e(z)-\varphi(z)| \ud z=0.
\]
\end{remark}

Hypothesis~\eqref{eq:fourier} with the help of the Fourier Inversion Theorem yields the existence of continuous and bounded densities $f_\e$ and $f$ such that
\begin{equation}\label{ec:densidades}
\begin{split}
f_\e(x)=\frac{1}{(2\pi)^d}\int_{\mathbb{R}^d} e^{-\ii \langle z,x \rangle}\varphi_\e(z)\ud z \quad \textrm{ and } \quad f(x)=\frac{1}{(2\pi)^d}\int_{\mathbb{R}^d} e^{-\ii \langle z,x \rangle}\varphi(z)\ud z,
\end{split} 
\end{equation}
 see Proposition~2.5 Item~(xii) 
 in~\cite{Sato1999}.
Due to the boundeness of $f_{\e}$ and $f$, for all $p\in[1,\infty]$ we have
\begin{equation}
f_\e,f \in L^p(\mathbb{R}^d)\quad 
\textrm{ for each }\quad \e>0.
\end{equation}

\begin{proof}[Proof of Proposition~\ref{lem:fouriertv}]
Recall $f_\e$ and $f$ given in~\eqref{ec:densidades}.
For all $x\in \mathbb{R}^d$ we have
\begin{equation}\label{eq:split}
\begin{split}
(2\pi)^d|f_\e(x)-f(x)|&=\left|\int_{\mathbb{R}^d} \left(e^{-\ii \langle z,x \rangle}\varphi_\e(z)-e^{-\ii \langle z,x \rangle}\varphi(z)\right)\ud z \right|\\
&\lqq 
\int_{\mathbb{R}^d} \left|\varphi_\e(z)-\varphi(z)\right|\ud z \\
&=\int_{\{|z|\lqq \ell\}} \left|\varphi_\e(z)-\varphi(z)\right|\ud z+\int_{\{|z|> \ell\}} \left|\varphi_\e(z)-\varphi(z)\right|\ud z
\end{split} 
\end{equation}
for all $\ell>0$.
For each $\ell>0$,  Item~(i) of Theorem~15.24 in~\cite{Klenke2020} yields
\begin{equation}\label{eq:R}
\lim_{\e\to 0^+}\int_{\{|z|\lqq \ell\}} \left|\varphi_\e(z)-\varphi(z)\right|\ud z=0.
\end{equation}
By~\eqref{eq:split} and~\eqref{eq:R} it follows that
\begin{equation}\label{eq:sides}
\begin{split}
(2\pi)^d \limsup\limits_{\e \to 0^+}|f_\e(x)-f(x)|& \lqq \limsup\limits_{\e \to 0^+}\int_{\{|z|> \ell \}} \left|\varphi_\e(z)-\varphi(z)\right|\ud z\\
&\lqq 
\limsup\limits_{\e \to 0^+}\int_{\{|z|> \ell\}} \left|\varphi_\e(z)\right|\ud z+ \int_{\{|z|> \ell\}} \left|\varphi(z)\right|\ud z
\end{split}
\end{equation}
for $x\in \mathbb{R}^d$ and any $\ell>0$.
Since $\varphi\in L^1(\mathbb{R}^d)$, it follows that 
\begin{equation}\label{eq:cola}
\lim\limits_{\ell\to \infty}\int_{\{|z|> \ell\}} \left|\varphi(z)\right|\ud z=0.
\end{equation}
Sending $\ell$ to infinity both sides of~\eqref{eq:sides}, Hypothesis~\eqref{tail} 
and~\eqref{eq:cola} imply
\begin{equation*}
\begin{split}
\limsup\limits_{\e \to 0^+}|f_\e(x)-f(x)|=0\quad \textrm{ for }\quad x\in \mathbb{R}^d.
\end{split}
\end{equation*}
Hence, it follows that
\[
\lim\limits_{\e \to 0^+}f_\e(x)=f(x)\quad \textrm{ for }\quad x\in \mathbb{R}^d.
\]
The preceding limit with the help of Lemma~\ref{lem:scheffe} in Appendix~\ref{apendtv} 
implies~\eqref{eq:tvlimit}. The proof is complete.
\end{proof}

\begin{proposition}[Parseval--Plancherel approach to the convergence in the total variation distance]\label{prof:parseval}
Let $(X_\e)_{\e>0}$ and $X$ be as in  Proposition~\ref{lem:fouriertv}. Assume that the Item~(i) in Proposition~\ref{lem:fouriertv} holds true.
In addition, assume that
\begin{itemize}
\item[(ii')] $L^2(\mathbb{R}^d)$- integrability:
\begin{equation}\label{eq:L2}
\varphi_\e,\varphi \in L^2(\mathbb{R}^d)\quad \textrm{ for each }\quad \e>0,
\end{equation}
where $\varphi_\e$ and $\varphi$ denote the characteristic function of $X_\e$ and $X$, respectively.
\item[(iii')] Tail behavior ($L^2(\mathbb{R}^d)$-uniform integrability):
\begin{equation}\label{tailL2}
\limsup_{\ell \to \infty}\limsup_{\e\to 0^+}\int_{\{|z|\gqq \ell\}}|\varphi_\e(z)|^2\ud z=0.
\end{equation}
\end{itemize}
Then it follows that
\begin{equation}\label{eq:tvlimitL2}
\lim\limits_{\e\to 0^+} \tv(X_\e,X)=0.
\end{equation}
\end{proposition}

\begin{remark}[$L^{2}(\mathbb{R}^d)$-convergence of the Fourier transform]\label{rem:L2}
Assume the hypothesis of Proposition~\ref{prof:parseval}.
Following step by step the proof of
Proposition~\ref{prof:parseval}
one can see that
\[
\lim\limits_{\e\to 0^+}\int_{\mathbb{R}^d} |\varphi_\e(z)-\varphi(z)|^2 \ud z=0.
\]
In addition, since $\varphi_{\e}$ and $\varphi$ are uniformly bounded by one, we have that Hypothesis (ii) and Hypothesis (iii) in Proposition~\ref{lem:fouriertv} implies Hypothesis (ii') and Hypothesis (iii') in Proposition~\ref{prof:parseval}. In general, Proposition~\ref{lem:fouriertv} implies the same result in $L^p(\mathbb R^d)$, for all $p\in(0,\infty)$. The reciprocal is in general false.
We stress that for the weaker assumption on the characteristic function
\begin{equation}
\varphi \in L^p(\mathbb{R}^d)\quad 
\textrm{ for some }\quad p>2,
\end{equation}
the
corresponding distribution function may be singular, see Theorem~13.4.2 in~\cite{Kawata}. Therefore, the   integrability assumption 
\eqref{eq:L2}  guarantees that the distribution function is absolutely continuous and hence it allows us to avoid singular distributions.
\end{remark}

\begin{remark}[Convergence in total variation distance without $L^{2}(\mathbb{R})$-integrability of the characteristic function]\label{rem:examplet}
While Proposition~\ref{lem:fouriertv}
and 
Proposition~\ref{prof:parseval}
provide sufficient conditions for convergence in the total variation distance, their conditions are not be strictly necessary. To illustrate this, we consider the following example.
Let $r\in (0,1/2)$ and $(\theta_n)_{n\in \mathbb{N}}$ be a sequence of positive numbers such that $\theta_n\to \theta>0$ as $n\to \infty$.
For each $n\in \mathbb N$ let $X_n\stackrel{d}{=}
 \Gamma(r,\theta_n)$ and $X\stackrel{d}{=} \Gamma(r,\theta)$, where $\Gamma(r,\vartheta)$ is the Gamma distribution having parameters $r>0$ and $\vartheta>0$ whose characteristic function is given in~\eqref{eq:deffourier}. That is, the corresponding densities are given, respectively, by $f_n(x)=\frac{\theta^r_n}{\Gamma(r)}x^{r-1}e^{-\theta_n x}\mathbf{1}_{(0,\infty)}(x)$ and $f(x)=\frac{\theta^r}{\Gamma(r)}x^{r-1}e^{-\theta x}\mathbf{1}_{(0,\infty)}(x)$.
By Scheff\'e's Lemma we have
\[
\lim\limits_{n\to\infty}\tv\left(X_n, X\right)=0.
\]
The corresponding characteristic functions are given by
\begin{equation}
\varphi_n(u)=\left(1-\ii \theta^{-1}_nu \right)^{-r},
\quad u\in \mathbb{R}
\end{equation}
and
\begin{equation}
\varphi(u)=\left(1-\ii \theta^{-1}u \right)^{-r},
\quad u\in \mathbb{R},
\end{equation}
respectively. Clearly, $\varphi_n,\varphi\not\in L^2(\mathbb{R})$ neither $\varphi_n,\varphi\not\in L^1(\mathbb{R})$.
\end{remark}

\begin{proof}[Proof of
Proposition~\ref{prof:parseval}]
Item~(ii') with the help of Lemma~9.2.8 in~\cite{Arturo} yields 
that the laws of $X_\e$ and $X$ are absolutely continuous with respect to the Lebesgue measure on $\mathbb{R}^d$, i.e., 
the existence of densities $f_\e$ and $f$ corresponding to $\varphi_\e$ and $\varphi$, respectively.
By Parseval--Plancherel Theorem we have
\begin{equation}\label{eq:nada}
\int_{\mathbb{R}^d} (f_\e(x)-f(x))^2 \ud x=
\int_{\mathbb{R}^d} |\varphi_\e(z)-\varphi(z)|^2 \ud z.
\end{equation}
We claim that the right-hand side of~\eqref{eq:nada} tends to zero as $\e\to 0^+$.
Indeed, let $\ell>0$ be fixed and note that
\begin{equation}
\int_{\mathbb{R}^d} |\varphi_\e(z)-\varphi(z)|^2 \ud z=
\int_{\{|z|\lqq \ell\}} |\varphi_\e(z)-\varphi(z)|^2 \ud z+\int_{\{|z|> \ell\}} |\varphi_\e(z)-\varphi(z)|^2 \ud z
\end{equation}
For each $\ell>0$,  Item~(i) of Theorem~15.24 in~\cite{Klenke2020} with the help of Item~(i) gives
\begin{equation}\label{eq:RL2}
\lim_{\e\to 0^+}\int_{\{|z|\lqq \ell\}} \left|\varphi_\e(z)-\varphi(z)\right|^2\ud z=0.
\end{equation}
Since
\begin{equation}
\begin{split}
\limsup_{\e\to 0^+}\int_{\{|z|> \ell\}} |\varphi_\e(z)-\varphi(z)|^2 \ud z
& \lqq 
2\limsup_{\e\to 0^+}\int_{\{|z|> \ell\}} |\varphi_\e(z)|^2 \ud z\\
&\qquad+2\int_{\{|z|> \ell\}} |\varphi(z)|^2 \ud z
\end{split}
\end{equation}
and $\varphi\in L^2(\mathbb{R}^d)$,
we have 
\begin{equation}
\begin{split}
\limsup\limits_{\ell\to \infty}\limsup_{\e\to 0^+}\int_{\{|z|> \ell\}} |\varphi_\e(z)-\varphi(z)|^2 \ud z
& \lqq 
2\limsup\limits_{\ell\to \infty}\limsup_{\e\to 0^+}\int_{\{|z|> \ell\}} |\varphi_\e(z)|^2 \ud z.
\end{split}
\end{equation}
Hence Item~(iii') gives
\begin{equation}
\lim\limits_{\ell\to \infty}\lim_{\e\to 0^+}\int_{\{|z|> \ell\}} |\varphi_\e(z)-\varphi(z)|^2 \ud z=0.
\end{equation}
By~\eqref{eq:nada} we obtain
\begin{equation}\label{eq:nadadosge}
\lim_{\e\to 0^+}\int_{\mathbb{R}^d} (f_\e(x)-f(x))^2 \ud x=0.
\end{equation}
To conclude we show that the preceding limit yields
\begin{equation}\label{eq:nadados}
\lim_{\e\to 0^+}\int_{\mathbb{R}^d} |f_\e(x)-f(x)|\ud x=0.
\end{equation}
Indeed,~\eqref{eq:nadadosge} with the help of
the Cauchy--Schwarz inequality gives 
for any $K>0$ 
\begin{equation}\label{eq:nadatres}
\lim_{\e\to 0^+}\int_{\{|x|\lqq K\}} |f_\e(x)-f(x)|\ud x=0.
\end{equation}
Recall that $f_\e$ and $f$ are the densities of the random vectors $X_\e$ and $X$, respectively.
Since $f\in L^1(\mathbb{R}^d)$ and $X_\e\to X$ in distribution as $\e\to 0^+$, we have
\begin{equation}\label{eq:nadacuatro}
\begin{split}
&\limsup_{K\to \infty}\limsup_{\e\to 0^+}\int_{\{|x|> K\}} |f_\e(x)-f(x)|\ud x\lqq 
\limsup_{K\to \infty}\limsup_{\e\to 0^+}\int_{\{|x|\gqq K\}} f_\e(x)\ud x\\
&\lqq 
\limsup_{K\to \infty}\limsup_{\e\to 0^+}\mathbb{P}(\|X_\e\|\gqq  K)\lqq 
\limsup_{K\to \infty}\mathbb{P}(\|X\|\gqq  K)=0.
\end{split}
\end{equation}
By~\eqref{eq:nadatres} 
and~\eqref{eq:nadacuatro} we 
deduce~\eqref{eq:nadados}.
\end{proof}

\section{\textbf{Scaling property and local central limit theorem for Gamma distribution}}\label{apen:gamma}

In this section, we show a scaling property for the second parameter of the Gamma distribution and a Gaussian local central limit theorem. We recall that for any positive parameters $\alpha$ and $\theta$, $\Gamma(\alpha,\theta)$ denotes a Gamma distribution with characteristic function given 
by~\eqref{eq:deffourier}. 

\begin{lemma}[Scaling property of the Gamma distribution]\label{lem:scapropgamma}
For any positive constants $\alpha$ and $\theta$ it follows that 
$\Gamma(\alpha,\theta)\stackrel{d}{=}\theta^{-1}\Gamma(\alpha,1)$.
\end{lemma}

\begin{proof}
For any $u\in \mathbb{R}$,~\eqref{eq:deffourier} yields
\begin{equation}\label{eq:chf}
\mathbb{E}\left[e^{\ii u (\theta^{-1}\Gamma(\alpha,1))}\right]=
\mathbb{E}\left[e^{\ii (u \theta^{-1})\Gamma(\alpha,1)}\right]=\left(1-\ii u\theta^{-1} \right)^{-\alpha}=\mathbb{E}\left[e^{\ii u \Gamma(\alpha,\theta)}\right].
\end{equation}

Recall that the distribution of a random variable is characterized by its characteristic function, see for instance Theorem~15.9 
in~\cite{Klenke2020}. 
Then~\eqref{eq:chf} implies the statement.
\end{proof}

\begin{lemma}[Local central limit theorem]\label{lem:lclt} 
For any positive $\alpha$ set
$\mathcal{Z}_\alpha:=\frac{\Gamma(\alpha,1)-\alpha}{\sqrt{\alpha}}$.
Then $\mathcal{Z}_\alpha$ converges in the total variation distance to a standard Gaussian distribution as $\alpha \to \infty$.
\end{lemma}

\begin{proof}
Let $\alpha>0$ be fixed and let $g_\alpha$ be the density of the random variable $\mathcal{Z}_\alpha$. In other words,
\begin{equation}\label{e:y0}
g_\alpha(z)=
\frac{\sqrt{\alpha}}{\Gamma(\alpha)}(\sqrt{\alpha}z+\alpha)^{\alpha-1}e^{-\sqrt{\alpha}z-\alpha}\mathbf{1}_{\{z\gqq -\sqrt{\alpha}\}},
\end{equation}
where $\Gamma$ denotes the usual Gamma function.
Recall that, for short, we write $F \stackrel{\alpha\to \infty}{\sim} G$ in a place of $\lim\limits_{\alpha \to \infty} \frac{F(\alpha)}{G(\alpha)} = 1$.
The Stirling formula for the function $\Gamma$ yields
\begin{equation*}
\Gamma(\alpha)\stackrel{\alpha\to \infty}{\sim}\sqrt{2\pi (\alpha-1)} \left( \frac{\alpha-1}{e} \right)^{\alpha-1}.
\end{equation*} 
Then we have 
\begin{equation}\label{e:y1}
\frac{\sqrt{\alpha}}{\Gamma(\alpha)}\stackrel{\alpha\to \infty}{\sim}
\frac{1}{\sqrt{2\pi}}\frac{\sqrt{\alpha}}{\sqrt{\alpha-1}}
\left( \frac{e}{\alpha-1} \right)^{\alpha-1}\stackrel{\alpha\to \infty}{\sim} \frac{1}{\sqrt{2\pi}}
\left( \frac{e}{\alpha-1} \right)^{\alpha-1}.
\end{equation}
Note that 
\begin{equation}\label{e:y2}
e\stackrel{\alpha\to \infty}{\sim}\frac{1}{\left(1-\frac{1}{\alpha} \right)^{\alpha-1}}\stackrel{\alpha\to \infty}{\sim}\left(\frac{\alpha}{\alpha-1} \right)^{\alpha-1}.
\end{equation}
Hence, for any $z\in \mathbb{R}$ fixed, combining~\eqref{e:y1},~\eqref{e:y2} in~\eqref{e:y0}  yields
\begin{equation}\label{eq:ju1}
\begin{split}
g_\alpha(z)&\stackrel{\alpha\to \infty}{\sim}
\frac{e^{-1}}{\sqrt{2\pi}}
\left(\frac{\alpha}{\alpha-1} \right)^{\alpha-1}\left(1+\frac{\sqrt{\alpha}z}{\alpha} \right)^{-1}
\left(1+\frac{\sqrt{\alpha}z}{\alpha} \right)^{\alpha}e^{-\sqrt{\alpha}z}\mathbf{1}_{\{z\gqq -\sqrt{\alpha}\}}\\
&\stackrel{\alpha\to \infty}{\sim}
\frac{1}{\sqrt{2\pi}}
\left(1+\frac{z}{\sqrt{\alpha}} \right)^{\alpha}e^{-\sqrt{\alpha}z}.
\end{split}
\end{equation}
By Lemma~\ref{lem:logconv} in Appendix~\ref{ap:tools} we have,
\begin{equation}\label{eq:ju2}
\left(1+\frac{z}{\sqrt{\alpha}} \right)^{\alpha}e^{-\sqrt{\alpha}z}\stackrel{\alpha \to \infty}{\sim} e^{-z^2/2}.
\end{equation}
Consequently, for any $z\in \mathbb{R}$~\eqref{eq:ju1} 
and~\eqref{eq:ju2} imply
\[
\lim\limits_{\alpha\to \infty}g_\alpha(z)=
\frac{1}{\sqrt{2\pi}}e^{-z^2/2},
\]
which with the help of Lemma~\ref{lem:scheffe} in Appendix~\ref{apendtv} gives the statement.
\end{proof}

\section{\textbf{Ergodicity and characteristic function for the  SDE (\ref{SDE})
}}\label{apenddist}

In the section, we recall the strong ergodicity (total variation convergence) and the explicit formulas for the characteristic function and the moment generating function of~\eqref{SDE}. 
We recall that
\begin{equation*}
\begin{split}
q_\e+1:=\frac{2\fb}{\e^2},\quad c_\e(t):=\frac{2\fa}{\e^2}\frac{1}{1-e^{-\fa t}}
\quad \textrm{ for any }\quad t>0 
\quad \textrm{ and }\quad
c_\e(\infty):=\frac{2\fa}{\e^2}.
\end{split}
\end{equation*} 

\begin{lemma}[Ergodicity for the CIR model in the Feller regime]\label{lem:ergodicity}
For any $\e\in (0,\sqrt{2\fb})$ and $x\in [0,\infty)$ let $(X^\e_t(x))_{t\gqq 0}$ be the unique strong solution 
of~\eqref{SDE}.
Then there exists a random variable $X^\e_\infty$ such that for any, $x\in [0,\infty)$ 
it follows that
\begin{equation}\label{eq:TVlem}
\lim\limits_{t\to \infty}\tv(X^\e_t(x),X^\e_\infty)=0
\end{equation}
and 
\begin{equation}\label{eq:WAlem}
\lim\limits_{t\to \infty}\mathcal{W}_p(X^\e_t(x),X^\e_\infty)=0.
\end{equation}
In addition, $X^{\e}_\infty\stackrel{d}{=}\Gamma(q_\e+1,c_\e(\infty))$,
where $q_\e+1$ and $c_\e(\infty)$ are defined in~\eqref{eq:qepsx0}.
\end{lemma}
The proof of~\eqref{eq:TVlem} can be found for instance in Theorem~1.2 ]
of~\cite{Jinetalt2019}.
Using the Fourier method in Proposition~\ref{lem:fouriertv} in Appendix~\ref{apendtv}, an alternative proof of~\eqref{eq:TVlem} can be obtained.
In addition, the convergence in~\eqref{eq:WAlem} can be deduced using Theorem~7.12 in~\cite{Villani2009}.

\begin{lemma}[Analytic continuation of the characteristic function for the CIR process]\label{moment generating function}
For any $\e\in (0,\infty)$ and $x\in [0,\infty)$ let $(X^\e_t(x))_{t\gqq 0}$ be the unique strong solution 
of~\eqref{SDE}.
For each $t>0$ the complex characteristic function of the time marginal $X^\e_t(x)$
is given by 
\begin{equation}\label{eq:mgfX2}
\mathbb{E}\left[e^{z X_t^\e(x)}\right]=
\frac{1}{\left(1-\frac{z}{c_\e(t)}\right)^{\frac{2\fb}{\e^2}}}
\exp \left(\frac{z c_\e(t)xe^{-\fa t}}{c_\e(t)-z}\right)\mathbf{1}_{\left\{ \mathsf{Re}(z)< c_\e(t)\right\}},
\end{equation}
where  $c_\e(t)=\frac{2\fa}{\e^2}\frac{1}{(1-e^{-\fa t})}$ and $\mathsf{Re}(z)$ denotes the real part of the complex number $z$.
In other words,
\begin{equation}\label{eq:formula}
\mathbb{E}\left[e^{z X_t^\e(x)}\right]=\frac{1}{\left(1-\frac{z}{c_\e(t)}\right)^{\frac{2\fb}{\e^2}}}
\exp \left(\frac{z c^2_\e(t)xe^{-\fa t}}{c^2_\e(t)-z^2}\right)\exp \left(\frac{z^2 c_\e(t)xe^{-\fa t}}{c^2_\e(t)-z^2}\right)\mathbf{1}_{\left\{\mathsf{Re}(z)<c_\e(t)\right\}}.
\end{equation}
\end{lemma}
The proof can be found in Proposition~1.2.4 p.~7 of~\cite{Alfonsi}.

\begin{corollary}[Characteristic function for the CIR process]\label{characteristic function}
For any $\e\in (0,\infty)$ and $x\in [0,\infty)$ let $(X^\e_t(x))_{t\gqq 0}$ be the unique strong solution 
of~\eqref{SDE}.
For any $z\in \mathbb{R}$, 
the characteristic function and the moment generating function of the time marginal $X^\e_t(x)$
are given respectively by 
\begin{equation}\label{eq:chacfunct}
\begin{split}
\psi(z;t,\e, x)&:=\mathbb{E}\left[e^{\ii z X_t^{\e}(x)}\right]\\
&=\frac{1}{\left(1-\ii\frac{z}{c_\e(t)}\right)^{\frac{2\fb}{\e^2}}} \exp \left(\ii\frac{z c^2_\e(t)xe^{-\fa t}}{c^2_\e(t)+z^2}\right)\exp \left(-\frac{z^2 c_\e(t)xe^{-\fa t}}{c^2_\e(t)+z^2}\right)
\end{split}
\end{equation}
 and
\begin{equation}\label{eq:mgfX}
\begin{split}
\phi(z;t,\e, x)&:=\mathbb{E}\left[e^{z X_t^\e(x)}\right]\\
&=
\frac{1}{\left(1-\frac{z}{c_\e(t)}\right)^{\frac{2\fb}{\e^2}}}
\exp \left(\frac{z c^2_\e(t)xe^{-\fa t}}{c^2_\e(t)-z^2}\right)\exp \left(\frac{z^2 c_\e(t)xe^{-\fa t}}{c^2_\e(t)-z^2}\right)\mathbf{1}_{\left\{ z< c_\e(t)\right\}}.
\end{split}
\end{equation}
Moreover, 
\begin{equation*}
|\psi(z;t,\e,x)|\lqq \frac{1}{\left(1+\frac{z^2}{c^2_\e(t)}\right)^{\frac{\fb}{\e^2}}}\quad \textrm{ for any }\quad z\in \mathbb{R},
\end{equation*} 
which yields 
\begin{equation*}
\int_{\mathbb{R}}|\psi(z;t,\e,x)|\ud z<\infty\quad \textrm{ when }\quad  \e\in (0,\sqrt{2\fb}).
\end{equation*}
\end{corollary}

\begin{lemma}[Gaussian central limit theorem at the mixing time]\label{Lem:limit 1}
Let $x\in[0,\infty)$ and $r\in \mathbb{R}$. Then it follows that 
\begin{equation}\label{eq:formulat}
\lim\limits_{\e\to 0^+}\mathbb{E}\left[e^{zY_{t_{\e}+r\cdot\omega_\e}^\e(x)}\right]=\exp\left(\frac{z^2}{2}\right)\exp\left(  z \frac{\sqrt{2}}{\sqrt{\fb}}(\fa x-\fb)e^{-r}\right)\quad 
\textrm{ for any }\quad  z\in \mathbb{C},
\end{equation}
where
$Y_t^{\e}(x)=\frac{c_{\e}(\infty) X_t^{\e}(x)-\left(q_{\e}+1\right)}{\sqrt{q_{\e}+1}}$ for any $t\gqq 0$, and $t_\e$ and $\omega_\e$ are given by 
\[t_{\e}:=
\frac{1}{\fa}\ln\left(\frac{1}{\e}\right)\quad \textrm{ and }\quad \omega_{\e}:=\frac{1}{\fa}.
\]
\end{lemma}

\begin{proof}
Let $z\in \mathbb{C}$ such that
\begin{equation}\label{eq:re1}
\mathrm{Re}(z)<\sqrt{q_\e+1}\frac{c_\e(t_{\e}+r\cdot  \omega_\e)}{c_\e(\infty)}=
\frac{\sqrt{2\fb}}{\e}\frac{1}{1-\e e^{-r}}.
\end{equation}
We denote $\varphi_{\e}(z;x):= \mathbb{E}[e^{zY_{t_{\e}+r \omega_\e}^\e(x)}]$.
Straightforward computations with the help of Lemma~\ref{moment generating function}
in Appendix~\ref{apenddist}
yields
\begin{equation}\label{Eq:varphia}
\varphi_\e(z;x)=
E_1(\e)E_2(\e;x)E_3(\e;x),
\end{equation}
where
\begin{equation}\label{eq:E1a}
E_1(\e):=
\frac{e^{- z \sqrt{q_{\e}+1}}}{\left(1-\frac{c_\e (\infty)z}{c_\e(t_\e+r\cdot \omega_\e)\sqrt{q_\e+1}}\right)^{q_\e+1}},
\end{equation}
\begin{equation}\label{eq:E2b}
E_2(\e;x):=\exp \left(\frac{\frac{c_\e(\infty)z}{\sqrt{q_\e+1}} c^2_\e(t_\e+r\cdot \omega_\e)xe^{-\fa (t_\e+r\cdot \omega_\e)}}{c^2_\e(t_\e+r\cdot \omega_\e)-\frac{c^2_\e(\infty)z^2}{q_\e+1}}\right)
\end{equation}
and 
\begin{equation}\label{eq:E3c}
E_3(\e;x):=\exp \left(\frac{\frac{c^2_\e(\infty)z^2}{q_\e+1} c_\e(t_\e+r\cdot \omega_\e)xe^{-\fa (t_\e+r\cdot \omega_\e)}}{c^2_\e(t_\e+r\cdot \omega_\e)-\frac{c^2_\e(\infty)z^2}{q_\e+1}}\right).
\end{equation}
In the sequel, we compute the limiting behavior of~\eqref{eq:E1a}.
Let $\alpha_\e:=q_\e +1$, $u:=-z$,
\[
u_\e=-\frac{c_\e(\infty)}{c_\e(t_\e+r\cdot \omega_\e)}z=-(1-e^{-\fa(t_\e+r\cdot \omega_\e)})z=u-\e e^{-r}u,
\]
from this we observe that
$\sqrt{\alpha_\e}(u_\e-u)=  z\sqrt{2\fb}e^{-r}$.
By Lemma~\ref{lem:logconvmodificado} in Appendix~\ref{ap:tools} we have
\begin{equation}\label{I1a}
\lim\limits_{\e\to 0^+}E_1(\e)=e^{z^2/2}e^{-z \sqrt{2\fb}e^{-r}}. 
\end{equation}
Now, we compute the asymptotic behavior of~\eqref{eq:E2b}. Straightforward computation with the help of~\eqref{eq:qepsx0} yields,
\begin{equation}
\begin{aligned}
E_2(\e ; x) & =\exp \left( z x \frac{c_{\e}(\infty)}{\sqrt{q_{\e}+1}} \frac{c_{\e}^2\left(t_{\e}+r \cdot \omega_{\e}\right)}{c_{\e}^2(\infty)} \frac{e^{-\mathfrak{a}\left(t_{\e}+r \cdot \omega_{\e}\right)}}{\frac{c_{\e}^2\left(t_{\e}+r \cdot \omega_{\e}\right)}{c_{\e}^2(\infty)}-\frac{z^2}{q_{\e}+1}}\right) \\
& =\exp \left( z x \frac{c_{\e}(\infty)}{\sqrt{q_{\e}+1}} \frac{1}{(1-e^{-\fa(t_\e+r\cdot \omega_\e)})^2} \frac{e^{-\mathfrak{a}\left(t_{\e}+r \cdot \omega_{\e}\right)}}{\frac{1}{(1-e^{-\fa(t_\e+r\cdot \omega_\e)})^2}-\frac{z^2}{q_{\e}+1}}\right)\\
& =\exp \left( z x \frac{2 \mathfrak{a}}{\sqrt{2 \mathfrak{b}}} \frac{1}{\left(1-\e e^{-r}\right)^2} \frac{e^{-r}}{\frac{1}{\left(1-\e e^{-r}\right)^2}-\frac{\e^2 z^2}{2 \mathfrak{b}}}\right),
\end{aligned}
\end{equation}
which implies
\begin{equation}\label{E2b}
\lim_{\e \rightarrow 0^{+}} E_2(\e ; x)=\exp \left( z x \frac{2 \mathfrak{a}}{\sqrt{2 \mathfrak{b}}} e^{-r}\right).    
\end{equation}
Finally, we compute the limiting behavior of~\eqref{eq:E3c}. We note that
\begin{equation*}
\begin{split}
E_3(\e;x)&=\exp \left(z^2 x \ \frac{1}{q_\e+1} \frac{ c_\e(t_\e+r\cdot \omega_\e)e^{-\fa (t_\e+r\cdot \omega_\e)}}{\frac{c^2_\e(t_\e+r\cdot \omega_\e)}{c^2_\e(\infty)}-\frac{z^2}{q_\e+1}}\right) \\
&=\exp \left(z^2 x \ \frac{\e^2}{2\fb} \frac{ \frac{2\fa}{\e}e^{-r}\frac{1}{1-\e e^{r}}}{\frac{1}{(1-\e e^{-r})^2}-\frac{z^2\e}{2\fb}}\right),
\end{split}
\end{equation*}
which yields
\begin{equation}\label{I3b}
\lim\limits_{\e\to 0^+}E_3(\e;x)=1.
\end{equation}
Combining~\eqref{I1a},~\eqref{E2b} and~\eqref{I3b} 
in~\eqref{Eq:varphia} 
with the observation that the right-hand side of~\eqref{eq:re1} tends to infinity as $\e\to 0^+$,
we obtain~\eqref{eq:formulat}.
\end{proof}

\begin{lemma}[Gaussian central limit theorem at any time scale]\label{Lemma lim 2}
Let $x=\frac{\fb}{\fa}$ and $r\in \mathbb{R}$. Then 
for any function $(s_\e)_{\e>0}$ such that 
$s_\e\to \infty$ as $\e\to 0^+$
it follows that 
\begin{equation}\label{eq:laplace1}
\lim\limits_{\e\to 0^+}\mathbb{E}\left[e^{zY_{s_{\e}}^\e(x)}\right]=\exp\left(\frac{z^2}{2}\right)\quad 
\textrm{ for any }\quad  z\in \mathbb{C},
\end{equation}
where
$Y_t^{\e}(x)=\frac{c_{\e}(\infty) X_t^{\e}(x)-\left(q_{\e}+1\right)}{\sqrt{q_{\e}+1}}$ for any $t\gqq 0$.
\end{lemma}

\begin{proof}
Let $z\in \mathbb{C}$ such that
\begin{equation}\label{eq:re1new}
\mathrm{Re}(z)<\sqrt{q_\e+1}\frac{c_\e(s_{\e})}{c_\e(\infty)}=
\frac{\sqrt{2\fb}}{\e}\frac{1}{1-e^{-\fa s_\e}},
\end{equation}
and let
$\widetilde{\varphi}_\e(z;x):=\mathbb E[e^ { z Y_{s_\e}^{\e }(x)}]$. 
Similarly to~\eqref{Eq:varphia} we obtain
\begin{equation}\label{Eq:varphinew}
\widetilde{\varphi}_\e(z;x)=
\widetilde{E}_1(\e)\widetilde{E}_2(\e;x)\widetilde{E}_3(\e;x),
\end{equation}
where
\begin{equation}\label{eq:E1newa}
\widetilde{E}_1(\e):=
\frac{e^{- z \sqrt{q_{\e}+1}}}{\left(1-\frac{c_\e (\infty)z}{c_\e(s_\e)\sqrt{q_\e+1}}\right)^{q_\e+1}}=\frac{1}{\left(1-\frac{(1-e^{-\fa s_\e})z}{\sqrt{q_\e+1}}\right)^{q_\e+1}}e^{- z \sqrt{q_{\e}+1}},
\end{equation}
\begin{equation}\label{eq:E2newa}
\begin{split}
\widetilde{E}_2(\e;x):&=\exp \left(\frac{\frac{c_\e(\infty)z}{\sqrt{q_\e+1}} c^2_\e(s_\e)xe^{-\fa s_\e}}{c^2_\e(s_\e)-\frac{c^2_\e(\infty)z^2}{q_\e+1}}\right)=\exp\left(z \sqrt{q_\e+1}e^{-\fa s_\e} \frac{1}{ 1-\frac{(1-e^{-\fa s_\e})^2 z^2}{q_\e+1}} \right)
\end{split}
\end{equation}
and 
\begin{equation}\label{eq:E3newa}
\widetilde{E}_3(\e;x):=\exp \left(\frac{\frac{c^2_\e(\infty)z^2}{q_\e+1} c_\e(s_\e)xe^{-\fa s_\e}}{c^2_\e(s_\e)-\frac{c^2_\e(\infty)z^2}{q_\e+1}}\right)=
\exp \left(z^2  \frac{e^{-\fa s_\e}\frac{1}{1-e^{-\fa s_\e}}}{\frac{1}{(1-e^{-\fa s_\e})^2}-\frac{z^2}{q_\e+1}}\right).
\end{equation}
In the sequel, we show that
\begin{equation}\label{eq:tildelimitea}
\lim\limits_{\e\to 0^+}\widetilde{\varphi}_\e(z;x)=e^{z^2/2}\quad \textrm{ for all }\quad z\in \mathbb{C}.
\end{equation}
By~\eqref{eq:E3newa} we observe that
\begin{equation}\label{eq:limitE3new}
\lim\limits_{\e\to 0^+} \widetilde{E}_3(\e;x)=1.
\end{equation}
Now, by~\eqref{eq:E1newa} we note that
\begin{equation}\label{eq:simplia}
\begin{split}
\widetilde{E}_1(\e)&=\frac{1}{\left(1-\frac{(1-e^{-\fa s_\e})z}{\sqrt{q_\e+1}}\right)^{q_\e+1}}
\frac{e^{z\sqrt{q_\e+1}(1-e^{-\fa s_\e})}}{e^{ z\sqrt{q_\e+1}(1-e^{-\fa s_\e})}}   e^{- z \sqrt{q_{\e}+1}}\\
&=\frac{1}{\left(1-\frac{(1-e^{-\fa s_\e})z}{\sqrt{q_\e+1}}\right)^{q_\e+1}}
\frac{1}{e^{ z\sqrt{q_\e+1}(1-e^{-\fa s_\e})}}   e^{- z\sqrt{q_\e+1}e^{-\fa s_\e}}.
\end{split}
\end{equation}
By~\eqref{eq:E2newa} and~\eqref{eq:simplia} we obtain
\begin{equation*}
\begin{split}
\widetilde{E}_1(\e)\widetilde{E}_2(\e;x)&=
\frac{1}{\left(1-\frac{(1-e^{-\fa s_\e})z}{\sqrt{q_\e+1}}\right)^{q_\e+1}}
\frac{1}{e^{ z\sqrt{q_\e+1}(1-e^{-\fa s_\e})}} \\
&\qquad\times 
\exp\left( -z \sqrt{q_\e+1}e^{-\fa s_\e}\frac{1}{q_\e+1} \left(\frac{z^2 (1-e^{-\fa s_\e})^2}{ 1+\frac{(1-e^{-\fa s_\e})^2 z^2}{q_\e+1}}\right) \right).
\end{split}
\end{equation*}
The preceding equality with the help of Lemma~\ref{lem:logconvmodificadodos} in Appendix~\ref{ap:tools} yields
\begin{equation}\label{eq:E2E3}
\lim\limits_{\e\to 0^+}
\widetilde{E}_1(\e)\widetilde{E}_2(\e;x)=e^{z^2/2}\quad \textrm{ for all }\quad z\in \mathbb{C}.
\end{equation}
Combining~\eqref{eq:limitE3new} 
and~\eqref{eq:E2E3} 
in~\eqref{Eq:varphinew} 
implies~\eqref{eq:laplace1}. 
\end{proof}

\section{\textbf{Second order Taylor's expansions}}\label{ap:tools}

In this section, we show refined asymptotics that were used along the manuscript.

\begin{lemma}[Second order asymptotics I]\label{lem:logconv}
For any $z\in \mathbb{R}$ it follows that
\begin{equation*}
\lim\limits_{\alpha \to \infty}\left(1+\frac{z}{\sqrt{\alpha}}
\right)^{\alpha}e^{-\sqrt{\alpha}z}= e^{-z^2/2}.
\end{equation*}
\end{lemma}

\begin{proof}
Let $z\in \mathbb{R}$ be fixed.
It is enough to show that
\[
\lim\limits_{\alpha \to \infty}\left[
\alpha\ln\left(1+\frac{z}{\sqrt{\alpha}} \right)-\sqrt{\alpha}z\right]=-\frac{z^2}{2}.
\]
By Taylor's expansion we have for $\alpha \gg 1$
\[
\ln\left(1+\frac{z}{\sqrt{\alpha}} \right)=\frac{z}{\sqrt{\alpha}}-\frac{z^2}{2\alpha}+H(z,\alpha),
\]
where
\begin{equation}\label{eq:Gexpan}
\limsup_{r\to \infty}(\alpha^{3/2}|H(z,\alpha)|)<\infty.
\end{equation}
Then we have
\begin{equation*}
\begin{split}
\alpha \ln\left(1+\frac{z}{\sqrt{\alpha}} \right)-\sqrt{\alpha}z=-\frac{z^2}{2}+\alpha H(z,\alpha).
\end{split}
\end{equation*}
By~\eqref{eq:Gexpan} we deduce
$\alpha|G(z,\alpha)|\to 0$ as $\alpha\to \infty$ and conclude the statement.
\end{proof}

\begin{lemma}[Second order asymptotics II]\label{lem:logconvmodificado}
Let $(u_\alpha)_{\alpha>0} \subset \mathbb{C}$ be a function such that 
\begin{equation}\label{lim:c}
\lim\limits_{\alpha\to \infty } \sqrt{\alpha}(u_\alpha-u)=c\quad \textrm{ for some }\quad u\in \mathbb{C}\quad \textrm{ and }\quad c\in \mathbb{C}.
\end{equation}
 Then it follows that
\begin{equation*}
\lim\limits_{\alpha\to \infty}\left(1+\frac{u_\alpha}{\sqrt{\alpha}} \right)^{\alpha}e^{-\sqrt{\alpha}u}=
e^{-u^2/2+c}.
\end{equation*}
\end{lemma}

\begin{proof}
It is enough to show that,
\[
\lim\limits_{\alpha \to \infty}\left[
\alpha\ln \left(1+\frac{u_\alpha}{\sqrt{\alpha}} \right)-\sqrt{\alpha}u\right]=-\frac{u^2}{2}+c.
\]
Here, we are considering the principal branch of  the complex logarithmic function $\log:=\ln$.
By Taylor's expansion we have for $\alpha \gg 1$
\[
\ln \left(1+\frac{u_\alpha}{\sqrt{\alpha}} \right)=\frac{u_\alpha}{\sqrt{\alpha}}-\frac{u^2_\alpha}{2\alpha}+H(u_\alpha,\alpha),
\]
where
\begin{equation}\label{eq:Gexpanmod}
\limsup_{\alpha \to \infty}(\alpha^{3/2}|H(u_\alpha,\alpha)|)<\infty.
\end{equation}
Then we have
\begin{equation*}
\begin{split}
\alpha \ln\left(1+\frac{u_\alpha}{\sqrt{\alpha}} \right)-\sqrt{\alpha}u_\alpha=-\frac{u^2_\alpha}{2}+\alpha H(u_\alpha,\alpha),
\end{split}
\end{equation*}
which can be written as
\begin{equation*}
\begin{split}
\alpha\ln\left(1+\frac{u_\alpha}{\sqrt{\alpha}} \right)-\sqrt{\alpha}u=\sqrt{\alpha}(u_\alpha-u)-\frac{u^2_\alpha}{2}+\alpha H(u_\alpha,\alpha).
\end{split}
\end{equation*}
By~\eqref{eq:Gexpanmod} we deduce that
$\alpha |H(u_\alpha,\alpha )|\to 0$ as $\alpha\to \infty$ and
with the help of~\eqref{lim:c} we 
 conclude the statement.
\end{proof}

\begin{lemma}[Second order asymptotics III]\label{lem:logconvmodificadodos}
Let $(u_\alpha)_{\alpha>0} \subset \mathbb{C}$ be a function such that 
\begin{equation}\label{lim:cdos}
\lim\limits_{\alpha \to \infty } u_\alpha=u\quad \textrm{ for some }\quad u\in \mathbb{C}.
\end{equation}
 Then it follows that
\begin{equation*}
\lim\limits_{\alpha \to \infty}\left(1+\frac{u_\alpha}{\sqrt{\alpha}} \right)^{\alpha}e^{-\sqrt{\alpha }u_\alpha}=
e^{-u^2/2}.
\end{equation*}
\end{lemma}

\begin{proof}
It is enough to prove that,
\[
\lim\limits_{\alpha \to \infty}\left[
\alpha \ln\left(1+\frac{u_\alpha }{\sqrt{\alpha}} \right)-\sqrt{\alpha}u_\alpha\right]=-\frac{u^2}{2}.
\]
Here, we are considering the principal branch of  the complex logarithmic function $\log:=\ln.$
By Taylor's expansion we have for $\alpha \gg 1$
\[
\ln\left(1+\frac{u_\alpha}{\sqrt{\alpha}} \right)=\frac{u_\alpha}{\sqrt{\alpha}}-\frac{u^2_\alpha }{2\alpha}+H(u_\alpha,\alpha),
\]
where
\begin{equation}\label{eq:Gexpandos}
\limsup_{\alpha\to \infty}(\alpha^{3/2}|H(u_\alpha,\alpha)|)<\infty.
\end{equation}
Then we have
\begin{equation*}
\begin{split}
\alpha\ln\left(1+\frac{u_\alpha}{\sqrt{\alpha}} \right)-\sqrt{\alpha}u_\alpha=-\frac{u^2_\alpha}{2}+\alpha H(u_\alpha,\alpha).
\end{split}
\end{equation*}
By~\eqref{eq:Gexpandos} we deduce that
$\alpha|H(u_\alpha,\alpha)|\to 0$ as $\alpha\to \infty$ and
with the help of~\eqref{lim:cdos} we 
 conclude the statement.
\end{proof}

\begin{remark}
We point out that~\eqref{lim:c} implies~\eqref{lim:cdos}. However, in general~\eqref{lim:cdos} does not imply~\eqref{lim:c}.
\end{remark}

\section{\textbf{Mixing times equivalence of asymptotic profile cut-off}}\label{mixtime}

In this section, we show Item~(4) of Proposition~\ref{proposition asymptotic}. For clarity,  we state it as a Lemma~\ref{lem:cutoffsta}.
The  proofs of Item~(1), Item~(2) and Item~(3) are analogous.

Lemma~\ref{lem:cutoffsta} establishes the asymptotic relation between the  $\eta$-mixing times and the  cut-off time in a general framework. 
We stress that this result is valid under mild assumptions on the profile function $G$.

\begin{lemma}[Asymptotics of mixing times and profile cut-off phenomenon]
\label{lem:cutoffsta}
Assume that~\eqref{eq:monotonic} holds true and
that the function 
$G:\mathbb{R}\to (0,\mathrm{Diam})$ is strictly decreasing and continuous. In addition, assume that 
$\lim\limits_{r\rightarrow -\infty}\,G(r)=\mathrm{Diam}\quad 
\textrm{ and } \quad \lim\limits_{r\rightarrow \infty}\,G(r)=0$,
and the existence of a function
$((t_{\e},\omega_{\e}))_{\e>0}$ satisfying
\[
\lim \limits_{\e \to 0^+}t_{\e}= \infty \quad
\textrm{ and } \quad 
\lim\limits_{\e\to 0^+}\frac{\omega_{\e}}{t_{\e}}=0.
\]
Then the following statements are equivalent.
\begin{itemize}
\item[(a)] 
The following limit
\begin{equation}\label{ec:limitperfil}
\lim\limits_{\e \rightarrow 0^+}\,\mathrm{dist}_\e(\mathrm{Law}\,\mathcal{X}^{\e}_{t_{\e}+r\cdot \omega_{\e}},\nu^\e)=G(r)\quad 
\textrm{ for any }\quad r\in \mathbb{R}
\end{equation}
holds true.
\item[(b)] For any $\eta\in (0,\mathrm{Diam})$
it follows that
\begin{equation}\label{ec:tmixp}
\tau^{\e,\mathrm{dist}_{\e}}_{\mathrm{mix}}(\eta)=t_\e+ G^{-1}(\eta)\cdot \omega_\e +\mathrm{o}(\omega_\e),\quad 
\textrm{ as }\quad \e\to 0^+,
\end{equation}
is valid,
where $G^{-1}(\eta)$ is the unique $r:=r_\eta\in \mathbb{R}$ such that $G(r)=\eta$ and the function 
$\mathrm{o}(\omega_\e)$ satisfies
$\frac{\mathrm{o}(\omega_\e)}{\omega_\e}\to 0$ as $\e\to 0^+$.
\end{itemize}
\end{lemma}

\begin{proof}
We start assuming  that the statement of Item~(a) is valid. 

Let $\eta \in (0,\mathrm{Diam})$ be fixed. 
We first show that
\begin{equation}\label{eq:primera}
\limsup\limits_{\e\to 0^+}\frac{\tau_{\text {mix }}^{\e,x}(\eta)-\left(t_{\e}+{G}^{-1}(\eta)\cdot \omega_{\e}\right)}{\omega_{\e}} \lqq 0.
\end{equation}
We take $\delta^*$ such that $0<\delta^*<\eta$.
Since $G$ is a bijection, there exists a unique $r^* \in \mathbb{R}$ such that $G(r^*)=\eta-\delta^*$, i.e, $r^*={G}^{-1}(\eta-\delta^*)$.
For the preceding choice  $\eta$ and  $\delta^*$,
the limit~\eqref{ec:limitperfil} with the help of~\eqref{diameter} gives the existence of $\e^*=\e^*(\eta, \delta^*)>0$ such that
\begin{equation}\label{eq:F3}
\begin{split}
0&<\eta<\mathrm{Diam}_\e\quad \textrm{ and }\\
\operatorname{dist}_{\e}\left(t_{\e}+r^* \cdot\omega_{\e}\right)&<\delta^*+G\left(r^*\right)=\eta
\end{split}
\end{equation}
for all $0<\e<\e^*$. 
Thus~\eqref{eq:F3} with the help of the definition of $\eta$-mixing time yields
\begin{equation}
\begin{split}
\tau_{\text {mix }}^{\e,x}(\eta) & \lqq t_{\e}+r^*\cdot \omega_{\e} =t_{\e}+G^{-1}(\eta-\delta^*)\cdot\omega_{\e}\\
& =t_{\e}+G^{-1}(\eta) \cdot\omega_{\e}+\left(G^{-1}(\eta-\delta^*)-G^{-1}(\eta)\right)\cdot \omega_{\e}
\end{split}
\end{equation}
for all $0<\e<\e^*$. Hence, we obtain
\begin{equation}\label{eq:superior}
\limsup\limits_{\e\to 0^+}\frac{\tau_{\text {mix }}^{\e,x}(\eta)-\left(t_{\e}+{G}^{-1}(\eta)\cdot \omega_{\e}\right)}{\omega_{\e}} \lqq {G}^{-1}(\eta-\delta^*)-{G}^{-1}(\eta)
\end{equation}
for any $\eta>0$ and $\delta^*\in (0,\eta)$. Since $G$ is a strictly decreasing continuous function, then $G^{-1}$ is also a continuous function, see Proposition~3.6.6. in~\cite{Lebl2023}. Moreover,  the left-hand of~\eqref{eq:superior} does not depend on $\delta^*$, then tending $\delta^*\to 0^+$ we deduce~\eqref{eq:primera}. 

We now prove that
\begin{equation}\label{eq:segunda}
\liminf\limits_{\e\to 0^+}\frac{\tau_{\text {mix }}^{\e,x}(\eta)-\left(t_{\e}+{G}^{-1}(\eta)\cdot \omega_{\e}\right)}{\omega_{\e}} \gqq 0.
\end{equation}
We take $\delta_*>0$ such that $\eta+\delta_*<\mathrm{Diam}$.
Since $G$ is a bijection, there exists a unique $r_* \in \mathbb{R}$ such that $G(r_*)=\eta+\delta_*$, i.e, $r_*={G}^{-1}(\eta+\delta_*)$.
For the preceding choice  $\eta$ and  $\delta_*$,
the limit~\eqref{ec:limitperfil} with the help of~\eqref{diameter} implies the existence of $\e_*=\e_*(\eta, \delta_*)>0$ such that
\begin{equation}\label{eq:Hh3}
\begin{split}
0&<\eta<\mathrm{Diam}_\e\quad \textrm{ and }\\
\eta&=-\delta_*+G\left(r_*\right)<
\operatorname{dist}_{\e}\left(t_{\e}+r_* \cdot\omega_{\e}\right)
\end{split}
\end{equation}
for all $0<\e<\e_*$. 
Thus~\eqref{eq:Hh3} with the help of the definition of $\eta$-mixing time yields
\begin{equation}
\begin{split}
\tau_{\text {mix }}^{\e,x}(\eta) & \gqq t_{\e}+r_* \cdot\omega_{\e} =t_{\e}+G^{-1}(\eta+\delta_*)\cdot\omega_{\e}\\
& =t_{\e}+G^{-1}(\eta)\cdot \omega_{\e}+\left(G^{-1}(\eta+\delta_*)-G^{-1}(\eta)\right)\cdot \omega_{\e}
\end{split}
\end{equation}
for all $0<\e<\e^*$, and therefore
\begin{equation}\label{eq:inferior}
\liminf\limits_{\e\to 0^+}\frac{\tau_{\text {mix }}^{\e,x}(\eta)-\left(t_{\e}+{G}^{-1}(\eta)\cdot \omega_{\e}\right)}{\omega_{\e}} \gqq  {G}^{-1}(\eta+\delta_*)-{G}^{-1}(\eta)
\end{equation}
for any $\eta>0$ and $\delta_*$ such that $\eta+\delta^*<\mathrm{Diam}$.
Since $G^{-1}$ is a continuous function and 
 the right-hand of~\eqref{eq:inferior} does not depend on $\delta_*$, tending $\delta_*\to 0^+$ we deduce~\eqref{eq:segunda}. 
By~\eqref{eq:primera} and~\eqref{eq:segunda} we obtain the statement of Item~(b).

In the sequel, we assume that the statement of Item~(b) is valid. 
Let $r\in \mathbb{R}$ be fixed. 

We start with the proof of 
\begin{equation}\label{eq:upper}
\limsup\limits_{\e\to 0^+}
\mathrm{dist}_\e(\mathrm{Law}\,\mathcal{X}^{\e}_{t_{\e}+r\cdot \omega_{\e}},\nu^\e)\lqq G(r).
\end{equation}
Assume that $\delta^*>0$ and take $\eta^{*}:=G(r-\delta^*)\in (0,\mathrm{Diam})$.
By hypothesis $G$ is a bijection and then we have that $r=G^{-1}(\eta^*)+\delta^*$. 
For the preceding choice of $\delta^*>0$ and $\eta^*\in (0,\mathrm{Diam})$,
the asymptotic~\eqref{ec:tmixp} with the help of~\eqref{diameter} yields the existence of $\e^*:=\e^*(r,\delta^*)>0$ such that 
\begin{equation}\label{ec:tmixp2}
\begin{split}
0&<\eta^*<\mathrm{Diam}_\e\quad 
\textrm{ and }\\
\tau^{\e,\mathrm{dist}_{\e}}_{\mathrm{mix}}(\eta^*)&<t_\e+ (G^{-1}(\eta^*)+\delta^*)\cdot \omega_\e
=t_\e+ r\cdot \omega_\e
\end{split}
\end{equation}
for all $\e\in (0,\e^*)$.
The definition of $\eta^*$-mixing time with the help 
of~\eqref{eq:monotonic} yields
\begin{equation}
\mathrm{dist}_\e(\mathrm{Law}\,\mathcal{X}^{\e}_{t_{\e}+r\cdot \omega_{\e}},\nu^\e)\lqq \eta^*=G(r-\delta^*)\quad \textrm{ for all }\quad \e\in (0,\e^*).
\end{equation}
Therefore, 
\begin{equation}\label{ec:limsup}
\limsup\limits_{\e\to 0^+}
\mathrm{dist}_\e(\mathrm{Law}\,\mathcal{X}^{\e}_{t_{\e}+r\cdot \omega_{\e}},\nu^\e)\lqq G(r-\delta^*).
\end{equation}
Since the left right-hand 
of~\eqref{ec:limsup} does not depend on $\delta^*$ and the function $G$ is continuous, tending $\delta^*\to 0^+$ we 
deduce~\eqref{eq:upper}.  

We continue with the proof of 
\begin{equation}\label{eq:lower}
\liminf\limits_{\e\to 0^+}
\mathrm{dist}_\e(\mathrm{Law}\,\mathcal{X}^{\e}_{t_{\e}+r\cdot \omega_{\e}},\nu^\e)\gqq  G(r).
\end{equation}
Assume that $\delta_*>0$ and take $\eta_{*}:=G(r+\delta_*)\in (0,\mathrm{Diam})$, i.e.,
$r=G^{-1}(\eta_*)-\delta_*$. 
For the preceding choice of $\delta_*>0$ and $\eta_*\in (0,\mathrm{Diam})$,
the asymptotic~\eqref{ec:tmixp} with the help of~\eqref{diameter} yields the existence of $\e_*:=\e_*(r,\delta_*)>0$ such that
\begin{equation}\label{ec:tmixplower}
\begin{split}
0&<\eta_*<\mathrm{Diam}_\e\quad 
\textrm{ and }\\
0&<t_\e+ r\cdot \omega_\e=t_\e+ (G^{-1}(\eta_*)-\delta_*)\cdot \omega_\e
<\tau^{\e,\mathrm{dist}_{\e}}_{\mathrm{mix}}(\eta_*)
\end{split}
\end{equation}
for all $\e\in (0,\e_*)$.
The definition of $\eta_*$-mixing time with the help 
of~\eqref{eq:monotonic} yields
\begin{equation}
\mathrm{dist}_\e(\mathrm{Law}\,\mathcal{X}^{\e}_{t_{\e}+r\cdot \omega_{\e}},\nu^\e)\gqq  \eta_*=G(r+\delta_*)\quad 
\textrm{ for all }\quad \e\in (0,\e_*).
\end{equation}
Therefore, 
\begin{equation}\label{ec:liminf}
\liminf\limits_{\e\to 0^+}
\mathrm{dist}_\e(\mathrm{Law}\,\mathcal{X}^{\e}_{t_{\e}+r\cdot \omega_{\e}},\nu^\e)\gqq  G(r+\delta_*).
\end{equation}
Since the left right-hand of~\eqref{ec:limsup} does not depend on $\delta_*$ and the function $G$ is continuous, tending $\delta_*\to 0^+$ we deduce~\eqref{eq:lower}. 

By~\eqref{eq:upper} and~\eqref{eq:lower} we obtain the statement of Item~(a).
\end{proof}

\section*{Declarations}
\noindent
\textbf{Acknowledgments.}
G.~Barrera would like to express his gratitude to the
Department of Mathematics and
Statistics at
University of Helsinki, Helsinki, (Helsinki, Finland) 
and the Instituto Superior T\'ecnico (Lisbon, Portugal) for all the facilities used along the realization of this work. 
G.~Barrera would also like to thank IMPA for support and hospitality during the 2023 Post-Doctoral Summer Program, where partial work on this paper was undertaken.
L.~Esquivel would like  to express her gratitude to the Department of Mathematics at Universidad del Valle,  (Cali, Colombia) and Department of Mathematical Sciences at Universidad de Puerto Rico (Mayag\"uez, Puerto Rico) for all the facilities used along the realization of this work. 

\noindent
\textbf{Ethical approval.} Not applicable.

\noindent
\textbf{Competing interests.} The authors declare that they have no conflict of interest.

\noindent
\textbf{Authors' contributions.}
Both authors have contributed equally to the paper. 

\noindent
\textbf{Availability of data and materials.} Data sharing not applicable to this article, as no data sets were generated or analyzed during the current study.

\noindent
\textbf{Funding.}
The research of G.~Barrera has been supported by the Academy of Finland,
via an Academy project (project No.~339228) and the Finnish Centre of Excellence in Randomness and STructures (project No.~346306).




\begin{thebibliography}{100}
\bibitem{Aldous}
\textsc{Aldous, D. and Diaconis, P.} (1986). 
Shuffling cards and stopping times.
\textit{Am. Math. Monthly} \textbf{93} (5) 333--348.

\bibitem{Alfonsi}\textsc{Alfonsi, A.} (2015).
\textit{Affine diffusions and related processes: simulation, theory and applications}.
Bocconi \& Springer Series \textbf{6}. Springer ChamBocconi University Press.

\bibitem{Ambrosio} 
\textsc{Ambrosio, L., Bru\'e, E. and Semola, D.} (2021).
\textit{Lectures on optimal transport}. 
Unitext \textbf{130}.
Springer Cham.

\bibitem{Avelin} 
\textsc{Avelin, B. and  Karlsson, A.} (2022).
Deep limits and a cut-off phenomenon for neural networks.
\textit{J. Mach. Learn. Res.} \textbf{23} Paper No. 191, 29 pp. 

\bibitem{Bakeretalt2018} 
\textsc{Baker, J., Chigansky, P., Hamza, K. and Klebaner, F.} (2018).
Persistence of small noise and random initial conditions.
\textit{Adv. in Appl. Probab.} \textbf{50} (A) 67--81.

\bibitem{Barrera2018} 
\textsc{Barrera, G.} (2018).
Abrupt convergence for a family of
Ornstein-Uhlenbeck processes.
\textit{Braz. J. Probab. Stat.} \textbf{32} (1) 188--199.

\bibitem{BAmax} 
\textsc{Barrera, G.} (2021).
Cutoff phenomenon for the maximum of a sampling of Ornstein--Uhlenbeck processes.
\textit{Statist. Probab. Lett.} \textbf{168} Paper No. 108954, 7 pp.
\bibitem{BCostaJara} 
\textsc{Barrera, G., Da Costa, C. and Jara, M.} (2022).
Gradual convergence for Langevin dynamics on a degenerate potential.
ArXiv 2209.11026.


\bibitem{BHChaos} 
\textsc{Barrera, G. and H\"ogele, M.A.} (2023).
Ergodicity bounds for stable Ornstein-Uhlenbeck systems in Wasserstein distance with applications to cutoff stability.
\textit{Chaos: An Interdisciplinary Journal of Nonlinear Science} \textbf{2023} (33) No. 11 113124 19 pp.

\bibitem{BHPWA}
\textsc{Barrera, G., H\"ogele, M.A. and Pardo, J.C.} (2021).
Cutoff thermalization for Ornstein-Uhlenbeck systems with small L\'evy noise in the Wasserstein distance.
\textit{J. Stat. Phys.} \textbf{184} (3) Paper No. 27, 54 pp.

\bibitem{BarreraEJP2021} 
\textsc{Barrera, G., H\"ogele, M.A. and Pardo, J.C.} (2021). 
The cutoff phenomenon in total variation for nonlinear Langevin systems with small layered stable noise.
\textit{Electron. J. Probab.} \textbf{26} Paper No. 119, 76 pp.

\bibitem{BPHGBM} 
\textsc{Barrera, G., H\"ogele, M.A. and Pardo, J.C.} 
(2022).
Cutoff stability of multivariate geometric Brownian motion.
ArXiv 2207.01666.

\bibitem{BPHWASD} 
\textsc{Barrera, G., H\"ogele, M.A. and Pardo, J.C.} (2024).
The cutoff phenomenon in Wasserstein distance for nonlinear stable Langevin systems with small L\'evy noise.
\textit{J. Dyn. Diff. Equat.} 
\textbf{36} (1)  251--278.

\bibitem{BHPPshell} 
\textsc{Barrera, G., H\"ogele, M.A., Pardo, J.C. and Pavlyukevich, P.} (2024).
Cutoff ergodicity bounds in Wasserstein distance for a viscous energy shell model with L\'evy noise.
\textit{J. Stat. Phys.} \textbf{191} (105) 24 pp.

\bibitem{BPHSPDE} 
\textsc{Barrera, G., H\"ogele, M.A. and Pardo, J.C.} (2023).
The cutoff phenomenon for the stochastic heat and wave equation subject to small L\'evy noise.
\textit{Stoch. Partial Differ. Equ. Anal. Comput.} \textbf{11} (3) 1164--1202.

\bibitem{BJJSP} 
\textsc{Barrera, G. and Jara, M.} (2016).
Abrupt convergence of stochastic small perturbations of one dimensional dynamical systems. 
\textit{J. Stat. Phys.} \textbf{163} (1) 113--138.

\bibitem{BJAAP}
\textsc{Barrera, G. and Jara, M.} (2020). 
Thermalisation for small random perturbations of dynamical systems.
\textit{Ann. Appl. Probab.} \textbf{30} (3) 1164--1208.

\bibitem{BarreraLiu2021} 
\textsc{Barrera, G. and Liu, S.} (2021).
A switch convergence for a small perturbation of a linear
recurrence equation. 
\textit{Braz. J. Probab. Stat.} \textbf{35} (2) 224--241.

\bibitem{BPEJP2020} 
\textsc{Barrera, G. and Pardo, J.C.} (2020).
Cut-off phenomenon for Ornstein--Uhlenbeck processes driven by L\'evy processes.
\textit{Electron. J. Probab.} \textbf{25} Paper No. 15, 33 pp.

\bibitem{BarreraJaviera2}
\textsc{Barrera, J., Bertoncini, O. and Fern\'andez, R.} (2009).
Abrupt convergence and escape behavior for birth and death chains.
\textit{J. Stat. Phys.} \textbf{137} (4) 595--623.

\bibitem{BarreraJaviera3} 
\textsc{Barrera, J., Lachaud, B. and Ycart, B.} 
(2006). 
Cut-off for n-tuples of exponentially converging processes.
\textit{Stochastic Process. Appl.} \textbf{116} (10) 1433--1446.

\bibitem{BarreraJaviera1} 
\textsc{Barrera, J. and Ycart, B.} (2014).
Bounds for left and right window cutoffs.
\textit{ALEA, Lat.Am. J. Probab. Math. Stat.} \textbf{11} (2) 445--458.

\bibitem{BOK10} 
\textsc{Bayati, B., Owahi, H. and Koumoutsakos, P.} (2010).
A cutoff phenomenon in accelerated stochastic simulations of chemical kinetics via flow averaging (FLAVOR-SSA).
\textit{J. Chem. Phys.} \textbf{133} (24) Paper No. 244117, 6 pp.

\bibitem{Bobkov 2013}
\textsc{Bobkov, S.G.} (2013). 
Entropic approach to E. Rio's central limit theorem for transport distance.
\textit{Statist. Probab. Lett.} \textbf{83} (7) 1644--1648.

\bibitem{Boos} 
\textsc{Boos, D.} (1985).
A converse to Scheffe's theorem.
\textit{Ann. Statist.} \textbf{13} (1) 423--427.

\bibitem{Bordenave}
\textsc{Bordenave, C., Caputo, P. and Salez, J.} (2019). 
Cutoff at the ``entropic time'' for sparse Markov chains.
\textit{Probab. Theory Related Fields} \textbf{173} (1-2) 261--292.

\bibitem{Boursier} 
\textsc{Boursier, J., Chafa\"i, D. and Labb\'e, C.} (2023).
Universal cutoff for Dyson Ornstein Uhlenbeck process. 
\textit{Probab. Theory Related Fields} \textbf{185} (1-2) 449--512.

\bibitem{Bresar} 
\textsc{Bre\v{s}ar, M. and Mijatovi\'c, A.} (2024).
Subexponential lower bounds for $f$-ergodic Markov processes. 
\textit{Probab. Theory Related Fields}.

\bibitem{Ceccheetalt}
\textsc{Ceccherini-Silberstein, T., Scarabotti, F. and Tolli, F.} (2008).
\textit{Harmonic analysis on finite groups.
Representation theory, Gelfand pairs and Markov chains}.
Cambridge Stud. Adv. Math. \textbf{108}
Cambridge University Press, Cambridge.

\bibitem{Chae}
\textsc{Chae, M. and Walker, S.G.} (2020).
Wasserstein upper bounds of the total variation for smooth densities.
\textit{Statist. Probab. Lett.} \textbf{163} Paper No. 108771, 6 pp.

\bibitem{Chenbook}
\textsc{Chen, L.H., Goldstein, L. and Shao, Q.-M.} (2011).
\textit{Normal approximation by Stein's method}.
Probab. Appl.
Springer Heidelberg.

\bibitem{Coxetalt1985} 
\textsc{Cox, J., Ingersoll, J. and Ross, S. }
(1985).
A theory of the term structure of interest rates.
\textit{Econometrica} \textbf{53} (2) 385--407.

\bibitem{DasGupta} 
\textsc{DasGupta, A.} (2008). 
\textit{Asymptotic theory of statistics and probability}.
Springer Texts Statist.
Springer New York.

\bibitem{Devroyeetalt2001} 
\textsc{Devroye, L. and  Lugosi, G. } (2001).
\textit{Combinatorial methods in density estimation}.
Springer Ser. Statist.
Springer-Verlag New York.

\bibitem{Diaconis1996}
\textsc{Diaconis, P.} (1996).
The cutoff phenomenon in finite Markov chains.
\textit{Proc. Nat. Acad. Sci. U.S.A.} \textbf{93} (4)  1659--1664.

\bibitem{PDIbook} 
\textsc{Diaconis, P.} (1988).
\textit{Group representations in probability and statistics}. 
Institute of Mathematical Statistics Lecture Notes
Monograph Series \textbf{11}. Institute of Mathematical Statistics, Hayward, CA.

\bibitem{Ding2009}
\textsc{Ding, J., Lubetzky, E. and Peres, Y.} (2009).
The mixing time evolution of Glauber dynamics for the mean-field Ising model.
\textit{Comm. Math. Phys.} \textbf{289} (2) 725--764.

\bibitem{Elboim2022}
\textsc{Elboim, D. and Schmid, D.} (2024).
Mixing times and cutoff for the TASEP in the high and low density phase.
\textit{Probab. Math. Phys.} \textbf{5} (2)  413--459.

\bibitem{EthierKurtz}
\textsc{Ethier, S.N. and Kurtz, T.G.} (1986). 
\textit{Markov processes. characterization and convergence}. 
Wiley Ser. Probab. Math. Statist. 
John Wiley \& Sons, Inc. New York.

\bibitem{Feller1951} 
\textsc{Feller, W.} (1951).
Two singular diffusion problems.
\textit{Ann. of Math.} \textbf{54} (1)  173--182.

\bibitem{Figalli}
\textsc{Figalli, A. and Glaudo, F.} (2023).
\textit{An invitation to optimal transport, Wasserstein distances, and gradient flows}.
EMS Textbk. Math.
EMS Press Berlin.

\bibitem{Friesen}
\textsc{Friesen, M., Jin, P., Kremer, J. and R\"udiger, B.} (2023). 
Exponential ergodicity for stochastic equations of nonnegative processes with jumps. 
\textit{ALEA, Lat. Am. J. Probab. Math. Stat.} \textbf{20} 593--627.

\bibitem{Fu}
\textsc{Fu, Z. and  Li, Z.} (2010). 
Stochastic equations of non-negative processes with jumps.
\textit{Stochastic Process. Appl.} \textbf{120} (3) 306--330.

\bibitem{Gantert20231}
\textsc{Gantert, N., Nestoridi, E. and Schmid, D.} (2023).
Mixing times for the simple exclusion process with open boundaries.
\textit{Ann. Appl. Probab.} \textbf{33} (2) 1172--1212.

\bibitem{Goncalves2023}
\textsc{Gon\c{c}alves, P., Jara, M., Menezes, O. and Marinho, R.} (2024+).
Sharp convergence to equilibrium for the SSEP with reservoirs. 
To appear in  \textit{Ann. Inst. Henri Poincar\'e Probab. Stat.} 


\bibitem{Ibragimov} 
\textsc{Ibragimov, I.D.}
(1956). 
On the composition of unimodal distributions.
\textit{Teor. Veroyatnost. i Primenen.} \textbf{1} 283--288.

\bibitem{Ikeda}
\textsc{Ikeda, N. and Watanabe, S.} (1981). 
\textit{Stochastic differential equations and diffusion processes}. 
North-Holland Publishing Co.

\bibitem{Jinetalt2019} 
\textsc{Jin, P., Kremer, J. and R\"udiger, B.}
(2019).
Moments and ergodicity of the jump-diffusion CIR process.
\textit{Stochastics} \textbf{91} (7) 974--997.

\bibitem{Jinetalt2013} 
\textsc{Jin, P., Mandrekar, V., R\"udiger, B. and Trabelsi, C.} (2013).
Positive Harris recurrence of the CIR process and its applications.
\textit{Commun. Stoch. Anal.} \textbf{7} (3) 409--424.

\bibitem{Kawata} 
\textsc{Kawata, T.} (1972). 
\textit{Fourier analysis in probability theory}.
Probab. Math. Statist. \textbf{15},
Academic Press, New York-London.

\bibitem{Kastoryano12} 
\textsc{Kastoryano, M.,  Reeb, D. and  Wolf, M.}  (2012). 
A cutoff phenomenon for quantum Markov chains.
\textit{J. Phys. A} \textbf{45} (7) Paper No. 075307, 16 pp.

\bibitem{Klenke2020} 
\textsc{Klenke, A.} (2020).
\textit{Probability theory. A comprehensive course}.
Universitext
Springer Cham.

\bibitem{Arturo}
\textsc{Kohatsu-Higa, A. and Takeuchi, A.} (2019). 
\textit{Jump SDEs and the study of their densities}.
Universitext
Springer Singapore.

\bibitem{Kusuoka} 
\textsc{Kusuoka, S. and Tudor, C.A.} (2012). 
Stein's method for invariant measures of diffusions via Malliavin calculus.
\textit{Stochastic Process. Appl.}  \textbf{122} (4) 1627--1651.

\bibitem{Labbe2022}
\textsc{Labb\'e, C. and  Petit, E.} (2024+).
Hydrodynamic limit and cutoff for the biased adjacent walk on the simplex.
To appear in \textit{Ann. Inst. Henri Poincar\'e Probab. Stat.}

\bibitem{Lachaud2005}
\textsc{Lachaud, B.} (2005).
Cut-off and hitting times of a sample of Ornstein--Uhlenbeck processes and its
average.
\textit{J. Appl. Probab.} \textbf{42} (4) 1069--1080. 

\bibitem{Lacoin2017}
\textsc{Lacoin, H.} (2017).
The simple exclusion process on the circle has a diffusive cutoff window.
\textit{Ann. Inst. Henri Poincar\'e Probab. Stat.} \textbf{53} (3)  
1402--1437.

\bibitem{Lebl2023}  
\textsc{Lebl, J.} (2023).
\textit{Basic analysis I. Introduction to real analysis, volume I.}
\url{https://www.jirka.org/ra/realanal.pdf}

\bibitem{Ledoux}
\textsc{Ledoux, M., Nourdin, I. and Peccati, G.} (2015). 
Stein's method, logarithmic Sobolev and transport inequalities.
\textit{Geom. Funct. Anal.} \textbf{25} (1) 256--306.

\bibitem{LeeRamilInsuk2023} 
\textsc{Lee, S., Ramil, M. and Seo, I.} (2023).
Asymptotic stability and cut-off phenomenon for the underdamped Langevin dynamics.
ArXiv 2311.18263.

\bibitem{Levin2009}  
\textsc{Levin, D., Peres, Y. and Wilmer, E.} (2009).
\textit{Markov chains and mixing times}. 
With a chapter by James G. Propp and D. Wilson. 
Amer. Math. Soc., Providence RI.

\bibitem{Ley}
\textsc{Ley, C., Reinert, G. and Swan, Y.} (2017). 
Stein's method for comparison of univariate distributions.
\textit{Probab. Surv.} \textbf{14} 1--52.
 
\bibitem{Li2020} 
\textsc{Li, Y. and Wang, S.} (2020).
Numerical computations of geometric ergodicity for stochastic dynamics. 
\textit{Nonlinearity} \textbf{33} (12)  6935--6970.

\bibitem{Lindvall}
\textsc{Lindvall, T.} (2002). 
\textit{Lectures on the coupling method}.
Dover Publications, Inc., Mineola NY.

\bibitem{Lubetzky2010}
\textsc{Lubetzky, E. and Sly, A.} (2010).
Cutoff phenomena for random walks on random regular graphs.
\textit{Duke Math. J.} \textbf{153} (3)  475--510. 

\bibitem{Lukacs} 
\textsc{Lukacs, E.} (1970).
\textit{Characteristic Functions}.
Second edition.
Hafner Publishing Co. New York.

\bibitem{Meliot}
\textsc{M\'eliot, P.L.} (2014). 
The cut-off phenomenon for Brownian motions on symmetric spaces of compact type.
\textit{Potential Anal.} \textbf{40} (4) 427--509.

\bibitem{Merle2019}
\textsc{Merle, M. and Salez, J.} (2019). 
Cutoff for the mean-field zero-range process.
\textit{Ann. Probab.} \textbf{47} (5) 3170--3201.

\bibitem{Meyn}
\textsc{Meyn, S.P. and Tweedie, R.L.} (1993). 
\textit{Markov chains and stochastic stability}.
Comm. Control Engrg. Ser.
Springer-Verlag London, Ltd. London.

\bibitem{Murray} 
\textsc{Murray, R. and Pego, R} (2016).
Cutoff estimates for the linearized Becker-D\"oring equations 
\textit{Commun. Math.
Sci.} \textbf{15} (6)  1685--1702. 

\bibitem{Nourdin}
\textsc{Nourdin, I. and Peccati, G.} (2012).
\textit{Normal approximations with Malliavin calculus: from Stein's method to universality}.
Cambridge Tracts in Math. \textbf{192}.
Cambridge University Press, Cambridge.

\bibitem{Oh}
\textsc{Oh, S. and Kais, S.} (2023).
Cutoff phenomenon and entropic uncertainty for random quantum circuits.
\textit{Electron. Struct.} \textbf{5} (3)  7 pp.

\bibitem{Panaretos2020}
\textsc{Panaretos, V.M. and Zemel, Y.} (2020). 
\textit{An invitation to statistics in Wasserstein space}.
SpringerBriefs Probab. Math. Stat.
Springer Cham.

\bibitem{Peyreeee}
\textsc{Peyr\'e, G. and Cuturi, M.} (2019).
\textit{Computational optimal transport: With applications to data science}.
Foundations and Trends in Machine Learning \textbf{11} (5-6).

\bibitem{Shen}
\textsc{Qi, W., Shen, Z. and Yi, Y}. (2024). Large deviation principle for quasi-stationary distributions and multiscale dynamics of absorbed singular diffusions. \textit{Probab. Theory Related Fields} \textbf{188} (1-2) 667--728.

\bibitem{Quattropani2023} 
\textsc{Quattropani, M. and Sau, F.} (2023).
Mixing of the averaging process and its discrete dual on finite-dimensional geometries.
\textit{Ann. Appl. Probab.} \textbf{33} (2)  1136--1171.

\bibitem{Reiss1989}
\textsc{Reiss, R.} (1989).
\textit{Approximate distributions of order statistics.
With applications to nonparametric statistics}. Springer Series in Statistics. Springer-Verlag New York.

\bibitem{Sato1999} 
\textsc{Sato, K.} (1999).
\textit{L\'evy processes and infinitely divisible distributions}.
Cambridge University Press, Cambridge.

\bibitem{Smith2017}
\textsc{Smith, A.} (2019). 
The cutoff phenomenon for random birth and death chains.
\textit{Random Structures Algorithms} \textbf{50} (2) 287--321.

\bibitem{Sweeting}
\textsc{Sweeting, T.J.} (1986). 
On a converse to Scheff\'e's theorem.
\textit{Ann. Statist.} \textbf{14} (3) 1252-1256.



\bibitem{Vernier20}
\textsc{Vernier, E.} (2020).
Mixing times and cutoffs in open quadratic fermionic systems.
\textit{SciPost Phys.} \textbf{9} (049) 1--30.

\bibitem{Villani2009} 
\textsc{Villani, C.} (2009). 
\textit{Optimal transport. Old and new}. Grundlehren der mathematischen Wissenschaften \textbf{338}. Springer-Verlag Berlin. 
\end{thebibliography}
\end{document}